\documentclass[11pt]{article}
\usepackage{amsmath, amsthm, amssymb, enumerate, lscape} 
\usepackage{fullpage}
\usepackage[top=1.5cm,left=1.5cm,right=1.5cm,bottom=1.5cm]{geometry}

\usepackage{makebox, relsize, empheq, mathtools, blkarray, arydshln, array, color, colortbl, xcolor, graphics}
\usepackage{float, units, booktabs, paralist, verbatim, subfig}
\usepackage{tikz}
\usepackage{comment}
\usepackage{longtable}

\theoremstyle{plain}
\newtheorem{theorem}{Theorem}
\newtheorem{lemma}[theorem]{Lemma}

\newtheorem{corollary}[theorem]{Corollary}
\newtheorem{problem}[theorem]{Problem}

\theoremstyle{definition}
\newtheorem{example}{Example}

\usepackage{array,color,colortbl}

\def\cref#1{Conjecture~$\ref{#1}$}
\def\Cref#1{Corollary~$\ref{#1}$}

\renewcommand{\geq}{\geqslant}
\renewcommand{\leq}{\leqslant}

\def\={\equiv}

\def\dfrac#1#2{\lower0.15ex\hbox{\large$\frac{#1}{#2}$}} 

\title{Mutually orthogonal frequency rectangles}

\author{Fahim Rahim\thanks{Department of Mathematics, The University of Waikato, Private Bag 3105, Hamilton 3240, New Zealand.} \thanks{Women Unviversity of AJK, Bagh, Azad Kashmir, Pakistan} \and Nicholas J. Cavenagh \footnotemark[1]}

\begin{document}

\maketitle

\begin{abstract}
A \emph{frequency rectangle} of type FR$(m,n;q)$ is an $m \times n$ matrix such that each symbol 
from a set of size $q$ appears $n/q$ times in each row and $m/q$ times in each column. 
Two frequency rectangles of the same type are said to be orthogonal if, upon superimposition, each possible ordered pair of symbols appear the same number of times. A set of $k$ frequency rectangles in which every pair is orthogonal is called a set of \emph{mutually orthogonal frequency rectangles}, denoted by $k$--MOFR$(m,n;q)$.
We show that a $k$--MOFR$(2,2n;2)$ and an orthogonal array OA$(2n,k,2,2)$ are equivalent. We also show that an OA$(mn,k,2,2)$ implies the existence of a $k$--MOFR$(2m,2n;2)$. We construct $(4a-2)$--MOFR$(4,2a;2)$ assuming the existence of  a Hadamard matrix of order $4a$.

A $k$--MOFR$(m,n;q)$ is said to be $t$--\emph{orthogonal}, if each subset of size $t$, when superimposed, contains each of the $q^t$ possible ordered $t$-tuples of entries exactly $mn/q^t$ times. A set of vectors over a finite field $\mathbb{F}_q$ is said to be \emph{$t$-independent} if each subset of size $t$ is linearly independent. We describe a method to obtain a set of $t$--orthogonal $k$--MOFR$(q^M, q^N, q)$ corresponding to a set of $t$--independent vectors in $(\mathbb{F}_q)^{M+N}$. We also discuss upper and lower bounds on the set of $t$--independent vectors and give a table of values for binary vectors of length $N \leq 16$.

A frequency rectangle of type FR$(n,n;q)$ is called a frequency square and a set of $k$ mutually orthogonal frequency squares is denoted by $k$--MOFS$(n;q)$ or $k$--MOFS$(n)$ when there is no ambiguity about the symbol set. For $p$ an odd prime, we show that there exists a set of $(p-1)$ binary MOFS$(2p)$, hence improving the lower bounds in (Britz et al. 2020) for the previously known values for $p \geq 19 $.
\end{abstract}



\section{Introduction}

A \emph{frequency rectangle} (also called \emph{F-rectangle}) of type FR$(m,n;q)$ is an $m \times n$ array based on a symbol set $S$ of size $q$, such that each element of $S$ appears exactly $n/q$ times in each row and $m/q$ times in each column. Two frequency rectangles, $F_1$ and $F_2$, of the same type, are said to be \emph{orthogonal} if each possible ordered pair of symbols appear the same number of times when $F_1$ and $F_2$ are superimposed. A set of $k$ frequency rectangles in which every pair is orthogonal is called a set of \emph{mutually orthogonal frequency rectangles}, denoted by $k$--MOFR$(m,n;q)$.

A \emph{frequency square} of type F$(n;q)$ is a frequency rectangle of type FR$(n,n;q)$. In the literature, a frequency square of type F$(n;q)$ is usually denoted by F$(n; \lambda)$, where $\lambda = n/q$ is the frequency of each symbol in each row or each column. However, we stick to the notation F$(n;q)$ where $q$ is the size of the symbol set to remain consistent with the rest of the notations used. The definition of orthogonality between two frequency squares is analogous to frequency rectangles. A set of frequency squares in which each pair is orthogonal is called a set of \emph{mutually orthogonal frequency squares} or MOFS, denoted by $k$--MOFS$(n;q)$ or simply by $k$--MOFS$(n)$ when there is no ambiguity about the symbol set. 

In the theory of frequency squares, most of the work has been dedicated to constructing the largest possible sets of MOFS. The upper bound, $k \leq (n-1)^2/(q-1)$, for a $k$--MOFS$(n;q)$ was first determined by Hedayat, Raghavarao, and Seiden \cite{hedayat1975further}.
The following theorem is a particular case of a more general result proved in \cite{Mandeli1984ontheconstruction}. However, we have included the proof here for thoroughness. The proof is similar  to the one given in \cite{hedayat1975further} for frequency squares.

\begin{theorem} \textup{\cite{Mandeli1984ontheconstruction}} \label{thm:upper_bound_rectangles}
	If a $k$--\textup{MOFR}$(m,n;q)$ exists, then:  
	\begin{gather}
	k \leq \frac{(m-1)(n-1)}{(q-1)}.
	\label{eq1}
	\end{gather}
\end{theorem}

\begin{proof}
	Let $ F_1, F_2, \dots , F_k $ be the elements of $k$--MOFR$(m,n;q)$. Corresponding to each $ F_\alpha $ we define an $ mn \times q $ matrix $ H_\alpha = (h_{(ij), \beta}) $, where $ i = 1,2, \dots, m; j=1,2, \dots , n $,   $ \beta $ runs over the symbols set, and 
	
$$
h_{(ij), \beta} =
	\begin{cases}
		1 \qquad \textit{if the entry in the } (i,j)^{th} \textit{ cell of } F_\alpha \textit{ is } \beta\\
		0 \qquad  otherwise.
	\end{cases}
$$ 
Observe that each column of $H_\alpha$ contains $mn/q$ 1's. Let $M$ be an $(mn) \times (kq)$ matrix defined as follows: \begin{gather*}
M = (H_1 | H_2 | \cdots |H_k).
\end{gather*} 

Since each row of $M$ corresponds to a fixed position $(i,j)$ of the set of frequency rectangles, by using the properties of frequency rectangles there are at least $(m-1) + (n-1) $ dependent rows in $M$. Therefore the rank of $M$,
\begin{gather*}
\textup{Rank}(M) \leq \min\{(m-1)(n-1) + 1, kq\}.
\end{gather*}

Observe that $ H_r^T H_s = (mn/q) I_q$ when $r = s$ and $ H_r^T H_s = (mn/q^2) J_q$ when $r \neq s$, where $ I_q $ is an identity matrix of order $ q $ and $J_q$ is a $q \times  q$ matrix of ones. 

Thus we have the following matrix of dimensions $(kq) \times (kq)$;

$$
M^TM = \begin{pmatrix}
n \lambda I_q  &  \lambda\lambda' J_q  &  \ldots  &  \lambda\lambda' J_q \\
\lambda\lambda' J_q &  n\lambda I_q   &  \ldots  &  \lambda\lambda' J_q  \\
\vdots        &    \vdots       &  \ddots  &   \vdots        \\
\lambda\lambda' J_q &  \lambda\lambda' J_q  &  \ldots  &   n\lambda I_q  \\
\end{pmatrix},
$$
where $\lambda = m/q$ and $ \lambda' = n/q$. The eigenvalues of $M^T M$ are 
$ n \lambda k, n \lambda $ and $ 0 $ with multiplicities $ 1, k(q-1) $ and $ k-1 $ respectively (see Appendix \ref{app:eigenvalues} for details). 
Since the sum of multiplicities of non-zero eigenvalues gives the rank of $ M^T M $,
\begin{gather*}
 kq-k+1 = R(M^TM)=R(M) \leq \min\{(m-1)(n-1) + 1, kq\},
\end{gather*} 
which gives the required result. 
\end{proof}

A $k$--MOFR$(m,n;q)$ or $k$--MOFS$(n;q)$ is said to be \emph{complete} if $k$ reaches the upper bound described in the above theorem. Complete sets of MOFR of type FR$(q^M, q^N;q)$ are known to exist when $q$ is a prime power \cite{federer1984pairwise}. For $q$ a prime power, Mandeli \cite{mandeli1992complete} describes a method to construct a complete set of MOFR$(q^M, 2q^N, q)$. For $m = 4a$ and $n = 4b$, Cheng \cite{cheng1980orthogonal} showed the existence of a complete set of MOFR$(m,n;2)$ provided that Hadamard matrices of order $4a$ and $4b$ exist. Also, assuming the existence of a Hadamard matrix of order $4b$, Federer, Hedayat, and Mandeli \cite{federer1984pairwise} describe a method to construct a complete set of MOFR$(2, 4b;2)$.

In Section \ref{sec:OA_FR}, we include results that further describe the relationship between Hadamard matrices, orthogonal arrays, and frequency rectangles. 
An \emph{orthogonal array} OA$(N,k,q,t)$ of strength $t$ is an $N \times k$ matrix of symbols based on a set of size $q$, such that in any $N \times t$ submatrix, each possible ordered $t$--tuple appears the same number of times as a row.
We show that an orthogonal array OA$(n,k,2,2)$ is equivalent to $k$--MOFR$(2,2n;2)$. We also show that if there exists an orthogonal array OA$(mn,k,2,2)$ then there exists $k$--MOFR$(2m,2n;2)$ and the existence of a Hadamard matrix of order $4a$ implies the existence of $(4a-2)$--MOFR$(4,2a;2)$.

 Complete sets of MOFS of type F$(q^N;q)$ are also known to exist when $q$ is a prime power \cite{mavron2000frequency, laywine2001table, li2014some, street1979generalized}. The complete set of MOFS of type F$(4a; 2)$ can be constructed by using a Hadamard matrix of order 4a \cite{federer1977existence}. However, no complete sets of MOFS for any other set of parameters are known to exist. In 2001, Laywine and Mullen \cite{laywine2001table} formulated a table of lower bounds for the maximum known values for the frequency squares of type F$(n;q)$ where $n \leq 100$. Later the table was improved by Li et al. \cite{li2014some} in 2014. Recently, in \cite{britz2020mutually}, the lower bounds in the case of $k$--MOFS$(n;q)$, where $n \equiv 2 \pmod{4}$ and $q = 2$ have been improved to $k \geq 17$ and it is also shown there that the complete sets do not exist for these parameters. In Section \ref{sec:(p-1)_binary_MOFS}, we give a method to construct a set of $(p-1)$--MOFS$(2p;2)$ where $p$ is an odd prime, thus improving the lower bounds in \cite{li2014some} and \cite{britz2020mutually} for such $p \geq 19$.

We next describe a stronger form of orthogonality for a set of frequency rectangles. A set $M$ of frequency rectangles of type FR$(m,n;q)$ is said to be $t$--\emph{orthogonal}, $t \geq 2$, if upon superimposition of any $t$ elements in $M$, each of the possible $q^t$ ordered $t$--tuples occurs the same number of times in the resulting array. By definition $k$--MOFR$(m,n;q)$ is $2$--orthogonal and any $t$--orthogonal set is also $t'$--orthogonal for any $2 \leq t' \leq t$.

Here we include an example to illustrate the definition further.

\begin{example} \label{exp:6_MOFR(4,4,2)}
Consider the set $M = \{F_1, F_2, \dots, F_6\}$, where each $F_i$ is given in Table \ref{tab:6MOFR_strength_3}. 
\begin{table}[H]
\begin{minipage}{0.16\textwidth}
\centering
\begin{tabular}{cccc}
0&	0&	1&	1\\
0&	0&	1&	1\\
1&	1&	0&	0\\
1&	1&	0&	0\\
\end{tabular}
\caption*{$F_1$}
\end{minipage}
\begin{minipage}{0.16\textwidth}
\centering
\begin{tabular}{cccc}
0&	1&	0&	1\\
0&	1&	0&	1\\
1&	0&	1&	0\\
1&	0&	1&	0\\
\end{tabular}
\caption*{$F_2$}
\end{minipage}
\begin{minipage}{0.16\textwidth}
\centering
\begin{tabular}{cccc}
0&	1&	0&	1\\
1&	0&	1&	0\\
1&	0&	1&	0\\
0&	1&	0&	1\\
\end{tabular}
\caption*{$F_3$}
\end{minipage}
\begin{minipage}{0.16\textwidth}
\centering
\begin{tabular}{cccc}
0&	1&	0&	1\\
1&	0&	1&	0\\
0&	1&	0&	1\\
1&	0&	1&	0\\
\end{tabular}
\caption*{$F_4$}
\end{minipage}
\begin{minipage}{0.16\textwidth}
\centering
\begin{tabular}{cccc}
0&	0&	1&	1\\
1&	1&	0&	0\\
1&	1&	0&	0\\
0&	0&	1&	1\\
\end{tabular}
\caption*{$F_5$}
\end{minipage}
\begin{minipage}{0.16\textwidth}
\centering
\begin{tabular}{cccc}
0&	0&	1&	1\\
1&	1&	0&	0\\
0&	0&	1&	1\\
1&	1&	0&	0\\
\end{tabular}
\caption*{$F_6$}
\end{minipage}
\caption{$3$--orthogonal $6$--MOFR$(4,4;2)$.}
\label{tab:6MOFR_strength_3}
\end{table}

Now if we superimpose any three elements of $M$, then each of the $2^3$ possible ordered 3--tuples occurs twice in the resultant array, as shown in Table \ref{tab:superimposition} for the case of $F_1, F_2, F_3$ and $F_1, F_4, F_6$. We leave it to the reader to verify that it is true for the rest of the cases. Thus $M$ is $3$--orthogonal. However, $M$ is  not $4$--orthogonal, since the sequences of odd weights do not occur when the arrays $F_1, F_2, F_3,$ and $F_5$ are superimposed (see Table \ref{tab:superimposition}).

\begin{table}[H]
\begin{minipage}{0.33\textwidth}
\centering
\begin{tabular}{cccc}
000&	011&	100&	111\\
001&	010&	101&	110\\
111&	100&	011&	000\\
110&	101&	010&	001\\
\end{tabular}
\caption*{$F_1,F_2,F_3$ }
\end{minipage}
\begin{minipage}{0.33\textwidth}
\centering
\begin{tabular}{cccc}
000&	010&	101&	111\\
011&	001&	110&	100\\
100&	110&	001&	011\\
111&	101&	010&	000\\
\end{tabular}
\caption*{$F_1, F_4, F_6$}
\end{minipage}
\begin{minipage}{0.33\textwidth}
\centering
\begin{tabular}{cccc}
0000&	0110&	1001&	1111\\
0011&	0101&	1010&	1100\\
1111&	1001&	0110&	0000\\
1100&	1010&	0101&	0011\\
\end{tabular}
\caption*{$F_1, F_2, F_3,F_5$}
\end{minipage}
\caption{}
\label{tab:superimposition}
\end{table}
\end{example}

In \cite{rahim2021row} and \cite{rahim2022row}, row-column factorial designs of strength $s$ are discussed. A row-column factorial design of strength $s$, denoted by $I_k(m,n,q,s)$, is an arrangement of $mn/q^k$ copies of the $q^k$--factorial design (that is, all the $k$-tuples over a set of size $q$) in an $m \times n$ array such that the elements in each row (column) forms an orthogonal array OA$(n,k,q,s)$ (OA$(m,k,q,s)$). By definition, the existence of an $I_k(m,n,q,s)$ (for any $s \geq 1$) implies the existence of  $k$--orthogonal $k$--MOFR$(m,n;q)$. 
 Since necessary and sufficient conditions are known for the existence of $I_k(m,n,q,1)$, we have the following theorem.

\begin{theorem} \textup{\cite{rahim2021row}}
    Let $m \leq n$. There exists $k$--orthogonal $k$--\textup{MOFR}$(m,n;q)$  if and only if:
     \begin{enumerate}
            \item[\textup{(i)}] $q \vert m$ and $q \vert n $;
            \item[\textup{(ii)}] if $k = q = m = 2$ then $n \equiv 0 \pmod{4}$; and
            \item[\textup{(iii)}] $(k,m,n,q) \neq (2,6,6,6)$.
    \end{enumerate}
\end{theorem}

A set of vectors over a field $\mathbb{F}_q$ is said to be $t$--\emph{independent} if each subset of size $t$ is linearly independent. In Section \ref{sec:t_orthogonal_FR}, we describe a relationship between a set of $t$--independent vectors over a finite field $\mathbb{F}_q$ and a set of    $t$--orthogonal MOFR$(q^M, q^N,q)$. We also exhibit a table that shows the maximum known values for $t$--independent vectors over $\mathbb{F}_2$ by using existence results for linear codes and some other known results in the literature. We also show that the existence of an orthogonal array OA$(2m,k,2,t)$ implies a $t$--orthogonal $k$--MOFR$(2m,2m,2)$.

\section{Orthogonal Arrays and Frequency Rectangles} \label{sec:OA_FR}

In this section, we use Hadamard matrices and orthogonal arrays to construct MOFR. A \textit{Hadamard matrix} $ H(n) $ is a square matrix of order $ n $, having entries from the set $ \{1,-1\} $ such that any two rows are orthogonal; that is it satisfies the equation:
\begin{equation*} \label{eq:hadamard_orthogonal}
H(n)H(n)^T=nI_n.
\end{equation*}

It has been conjectured that a Hadamard matrix of order $4n$ exists for each $n$ \cite{djokovic2007hadamard, hedayat1999orthogonal, horadam2012hadamard}. A Hadamard matrix with all the entries in its first column and first row equal to $1$ is called a \emph{normalized Hadamard matrix}. Any Hadamard matrix is equivalent to a normalized Hadamard matrix. A normalized Hadamard matrix has the following combinatorial properties.

\begin{lemma} \label{lem:hadamard_properties}
	Let $ {\bf{a}} = (a_1, \dots, a_n) $ and $ {\bf{b}}= (b_1, \dots, b_n) $ be two distinct rows, other than the first, of a normalized Hadamard matrix of order $ n $, $ n>2 $. Then
	\begin{enumerate}
		\item[\textup{(i)}] half of the entries $ a_i $ are $ +1's $ and half of them are $ -1's $.
		\item[\textup{(ii)}] the multiset $ \{(a_i,b_i): i=1,2,\dots, n\} $ contains each type of order pair exactly             $ n/4 $ times.
		\item[\textup{(iii)}] the conditions (i) and (ii) are also true for any two distinct  columns, other than the first, of a normalized Hadamard matrix.
	\end{enumerate}
\end{lemma}

A \emph{partial Hadamard matrix} or \emph{Hadamard rectangle} is a $k \times n$ matrix, having entries from the set $\{1, -1\}$ such that any two rows are orthogonal. Analogous to a normalized Hadamard matrix we can define a normalized Hadamard rectangle. The conditions (i) and (ii) in the Lemma \ref{lem:hadamard_properties} are also true for a normalized Hadamard rectangle, however, (iii) does not necessarily hold in the case of a rectangle. 

 An orthogonal array OA$(4a,4a-1,2,2)$ can be obtained from a normalized Hadamard matrix by removing the first column and replacing $-1$'s with 0's. The following lemma that describes the relationship between the two structures is a generalization of the Theorem 7.5 given in \cite[p.~148]{hedayat1999orthogonal} for the rectangular case. 

\begin{lemma} Let $k<2b$. There exists an \textup{OA}$(2b,k-1,2,2)$ if and only if there exists 
a $k\times 2b$ Hadamard rectangle. In particular, there exists an \textup{OA}$(2b,2b-1,2,2)$ if and only if there exists a Hadamard matrix $H(2b)$. 
\end{lemma}

If $B$ is a binary array then we define $\overline{B}$ to be the array obtained by interchanging $0$'s and $1$'s in $B$.

\begin{theorem}
\label{thm:ortho_to_mofrs}
Suppose there exists an \textup{OA}$(mn, k, 2,2)$. Then there exist $k$--\textup{MOFR}$(2m,2n;2)$. 
\end{theorem}

\begin{proof}
Let $M$ be an orthogonal array OA$(mn, k, 2,2)$. Let $ {\bf b} = (b_1, \dots , b_{mn})$ be any column of $M$. Define an $m \times n$ array $B$ corresponding to this column as follows:
$$
B = \left[
\begin{array}{cccc}
b_1 & b_2 & \dots & b_n \\
b_{n+1} & b_{n+2} & \dots & b_{2n} \\
\vdots & \vdots & \ddots & \vdots \\
b_{n(m-1)+1} & b_{n(m-1)+2} & \dots & b_{mn}
\end{array}
\right]
$$
Now, let $L_{\bf{b}}$ be the following array:

$$
L_{\bf{b}} =\left[
\renewcommand{\arraystretch}{1.2}
\begin{array}{c|c}
	 B & \overline{B} \\\hline
    \overline{B}  & B \\
\end{array}
\right]
$$
Then $L_b$ is a frequency rectangle of type FR$(2m,2n;2)$. Thus, by constructing an array corresponding to each column of $M$ we obtain a set of $k$ frequency rectangles of type FR$(2m, 2n; 2)$. The orthogonality of these arrays follows from the orthogonality of the columns of $M$.
\end{proof}

\begin{corollary}
Suppose there exists a Hadamard matrix $H(mn)$ where $4$ divides $mn$.  
Then there exist  
$(mn-1)$--\textup{MOFR}$(2m,2n;2)$.
\end{corollary}

\begin{lemma}
There exist $k$--MOFR$(2,2n;2)$ if and only if 
there exists an OA$(2n,k,2,2)$.  
\end{lemma}

\begin{proof}
Suppose there exists $k$--MOFR$(2,2n;2)$ and denote this set by $M$. Let $L_1, \dots, L_k \in M$. Let ${\bf r}_i$ be the first row of $L_i$. We claim that $[{\bf r}_1 \vert {\bf r}_2 \vert \dots \vert {\bf r}_k]$ is an OA$(2n,k,2,2)$. It is sufficient to show that ${\bf r}_1$ and ${\bf r}_2$ are orthogonal. Without loss of generality we may assume that $L_1$ has the following form:
$$
L_1 =
\begin{array}{|cccc | cccc|}
    \hline
	 0 & 0 & \dots & 0 & 1 & 1 & \dots &1 \\\hline
     1 & 1 & \dots & 1 & 0 & 0 & \dots &0 \\\hline
\end{array}
$$
Suppose that ${\bf r}_2$ contains $x$ zeros in the first $n$ positions. Then by the definition of frequency square the second row of $L_2$, that is $\overline{{\bf r}_2}$, contains $x$ zeros in the last $n$ positions. Therefore, the total number of $(0,0)$ pairs when $L_1$ and $L_2$ are superimposed is $2x$. Since $L_1$ and $L_2$ are orthogonal, $2x = n$ or $x = n/2$. Thus ${\bf r}_1$ and ${\bf r}_2$ when superimposed contain each type of pair the same number of times since $x = n-x$. \\
Conversely, corresponding to each column $\bf{c}$ of an OA$(2n,k,2,2)$ we generate a frequency square $L_{\bf{c}}$ which contains $\bf{c}$ and $\overline{\bf{c}}$ as its first and second row respectively. 
\end{proof}

\begin{theorem}
\sloppy
Suppose there exists a Hadamard matrix $H(4a)$. Then there exists $(4a-2)$--\textup{MOFR}$(4,2a;2)$.
\end{theorem}
 
\begin{proof}
Let $H$ be a Hadamard matrix of order $4a$ in normalized form. Replace $-1$'s with $0$'s in $H$. Let $c_i, i \in \{1, \dots, 4a\}$ be the columns of $H$. Without loss of generality we may assume that $c_1$ and $c_2$ have the following form:
$$
\begin{array}{ccc}
    c_1 && c_2 \\\hline
    1 && 1 \\
    1 && 1 \\
    \vdots && \vdots \\
    1 && 1 \\
    1 && 0 \\
    1 && 0 \\
    \vdots && \vdots \\
    1 && 0 \\
\end{array}
$$
Since the columns $c_2, \dots, c_{4a}$ of $H$ form an OA$(4a, 4a-1, 2,2)$, the columns $c_3, \dots, c_{4a}$ have the property that each contains exactly $a$ zeroes and $a$ ones in the first $2a$ positions. Let $ {\bf b} \in \{c_3, \dots, c_{2a}\}$, where $ {\bf b} = (b_1, \dots , b_{4a})$. Define a $2 \times 2a$ array $B$ corresponding to $\bf{b}$ as follows:
$$
B = \left[
\begin{array}{cccc}
b_1 & b_2 & \dots & b_{2a} \\
b_{2a+1} & b_{2a+2} & \dots & b_{4a} \\
\end{array}
\right]
$$
Now,
$$
L_{\bf{b}} =\left[
\renewcommand{\arraystretch}{1.2}
\begin{array}{c}
	 B \\\hline
    \overline{B} \\
\end{array}
\right]
$$
is a FR$(4,2a;2)$. Observe that the set $\{L_{\bf b} : {\bf b} \in \{c_3, \dots, c_{2a}\} \} $ forms a $(4a-2)$--MOFR$(4,2a;2)$.
\end{proof}

\section{$t$--orthogonal frequency rectangles} \label{sec:t_orthogonal_FR}

Recall that a $k$--MOFR$(m,n;q)$ is $t$--\emph{orthogonal} if each subset of size $t$, upon superimposition, gives $mn/q^t$ copies of the full factorial design. A set of vectors is said to be $t$--\emph{independent} if each subset of size $t$ is linearly independent. 
In this section, we  describe a relationship between a set of $t$--independent vectors and a set of $t$--orthogonal frequency rectangles. We also include some results from the literature about the known bounds for the size of a set of $t$--independent vectors. At the end of this section, we formulate a table that provides lower bounds on $k$ for a set of $t$--orthogonal $k$--MOFR$(m,n;2)$, by using existence results on binary linear codes.

Let ${\bf u} = (u_1, u_2, \dots, u_M)$ and ${\bf v} = (v_1, v_2, \dots, v_N)$ be vectors over the field $\mathbb{F}_q$ of length $M$ and $N$, respectively. Then we define ${\bf u} \oplus {\bf v} = (u_1, u_2, \dots, u_M, v_1, v_2, \dots, v_N) $ be the vector of length $M+N$ obtained by the concatenation of ${\bf u}$ and ${\bf v}$. And by a \emph{cyclic shift} of ${\bf v}$ we mean the vector ${\bf v}' = (v_N, v_1, v_2, \dots, v_{N-1})$.

We first give here a result that uses orthogonal arrays to construct a set of $t$--orthogonal frequency rectangles.
\begin{theorem}
If there exists an \textup{OA}$(2m,k,2,t)$ then there exists a $t$--orthogonal
$k$--\textup{MOFR}$(2m,2m;2)$. 
\end{theorem}

\begin{proof}
Let ${\bf v}$ be a column vector of an OA$(2m,k,2,t)$. Construct a frequency square $F_{\bf v}$ where column $i$ is the $i$th cyclic shift of ${\bf v}$. Clearly $F_{\bf v}$ is a frequency rectangle of type FR$(2m,2m;2)$.
Now the $t$--orthogonality of these arrays follows from the definition of $t$ in OA$(2m,k,2,t)$.
\end{proof}

The converse of the above theorem is not true in general. However, we have the following result. 

\begin{theorem}
    \sloppy If there exists a $t$--orthogonal $k$--\textup{MOFR}$(m,n;q)$, then there exists an orthogonal array \textup{OA}$(mn,k,q,t)$.
\end{theorem}

\begin{proof}
Let $S = \{F_1, \dots, F_k\}$ be a $t$--orthogonal $k$--\textup{MOFR}$(m,n;q)$. Corresponding to each $F_i \in S$, we construct a vector ${\bf f}_i$ of length $mn$ as follows. Let ${\bf v}_1, \dots, {\bf v}_m$ be, in sequential order, the row vectors of $F_i$. Let ${\bf f}_i = {\bf v}_1 \oplus {\bf v}_2 \oplus \dots \oplus {\bf v}_m $. Let $M$ be the $(mn) \times k $ array which contains each element of $\{ {\bf f}_i : 1 \leq i \leq k \} $ as a column vector. Observe that $M$ is an OA$(mn,k,q,t)$. 
\end{proof}

\begin{theorem} \label{thm:t_independent_FR_vectors}
\sloppy Let $S$ be a set of $k$ $t$-independent vectors in $({\mathbb F}_q)^{M+N}$ such that for each ${\bf{v}} = (v_1, \dots, v_{M+N}) \in S$
\begin{enumerate}
    \item[\textup{(i)}] $(v_1, \dots, v_M) \neq (0, \dots, 0) $,
    \item[\textup{(ii)}] $(v_{M+1}, \dots, v_{M+N}) \neq (0, \dots, 0) $,
\end{enumerate}
then there exists a $t$--orthogonal $k$--\textup{MOFR}$(q^M,q^N;q)$.
\end{theorem}

\begin{proof}
\sloppy Corresponding to each vector ${\bf v} \in S$ we construct a frequency rectangle $F_{\bf v}$ as follows. Let ${\bf v} = (v_1, \dots, v_{M+N})$. We define a polynomial, 
$$f_{\bf v} = \sum_{i = 1}^{M+N}v_ix_i. $$ 
Now we label the rows and columns of a $q^M \times q^N$ array by using all $M$--tuples and $N$--tuples, respectively, over the field $ \mathbb{F}_q$. Let the cell in the intersection of row $(r_1, \dots, r_M)$ and column $(c_1, \dots, c_N)$ contain the entry $f_{\bf v}(r_1, \dots, r_M, c_1, \dots, c_N) $.

To show that $F_{\bf v} $ is a frequency rectangle, fix a column ${\bf c}$ of $F_{\bf v}$, labeled by $(c_1, \dots, c_N)$. Let $\beta \in \mathbb{F}_q$. The number of appearances of $\beta$ in ${\bf c}$ is equal to the number of solutions to the following equation over $\mathbb{F}_q$:
\begin{gather} \label{eq:polynomial_in_column}
		f_{\bf v}(x_{1},..., x_{M},c_1, \dots, c_N) = \beta
		\end{gather}
By (i) there is at least one $i \in \{1, \dots, M\}$ for which $v_i \neq 0$, thus equation (\ref{eq:polynomial_in_column}) has exactly $q^{M-1}$ solutions over $\mathbb{F}_q$. This shows that $\beta$ occurs exactly $q^{M-1}$ times in ${\bf c}$. Similarly, we can show that each element of $\mathbb{F}_q$ occurs $q^{N-1}$ times in each row of $F_{\bf v}$.
 
Thus $\{F_{\bf v} : {\bf v} \in S \}$ is a set of $k$ frequency rectangles of type FR$(q^M, q^N;q)$. It remains to show that this set is $t$--orthogonal. Let $S'$ be a subset of $S$ of size $t$ and consider a $t$-tuple $\alpha = (\alpha_1, \dots, \alpha_t) $ in $\mathbb{F}_q$. Now consider the following system of equations:
$$
HX = \alpha^T
$$
where, $X = (x_1, \dots, x_{M+N})^T $ and $H$ is a $t \times (M+N)$ matrix that contains each element of $S'$ as a row. Since $S$ is a $t$--independent set of vectors, this system of equations has rank $t$ and therefore there are exactly $q^{M+N-t}$ solutions for each $\alpha \in (\mathbb{F}_q)^t$. Thus the set $\{F_{\bf v} : {\bf v} \in S \}$ is $t$--orthogonal $k$--MOFR$(q^M, q^N;q)$.
\end{proof}

\begin{corollary} \label{cor:t_indpendent_MOFR}
\sloppy Let $t \geq 1$. If there exists a set of $k$ $t$--independent vectors in $({\mathbb F}_q)^M$ then there exists a $t$--orthogonal $k$--\textup{MOFR}$(q^M,q^N;2)$, where $N\geq M$. 
\end{corollary}

\begin{proof}
   \sloppy Let $S$ be a set of $k$ $t$--independent vectors. Since $t \geq 1$, ${\bf 0} \not \in S $. For each ${\bf v} = (v_1, \dots, v_M) \in S$, we define ${\bf v}' = {\bf v} \oplus {\bf v} \oplus {\bf 0}$ of length $M+N$, where ${\bf 0} $ is a zero vector of length $N-M$. Let $S' = \{ {\bf v}' : {\bf v} \in S \}$. Observe that $S'$ is $t$--independent and satisfies the conditions (i) and (ii) of Theorem \ref{thm:t_independent_FR_vectors}.
\end{proof}

\begin{theorem}
Suppose there exists a set of $k_1$ $3$--independent vectors of length $M$ and a set of $k_2$  $3$--independent vectors of length $N$ over the field $\mathbb{F}_q$. Then there exists $t$--orthogonal $(k_1k_2)$--\textup{MOFR}$(q^M, q^N; q)$.      
\end{theorem}

\begin{proof}
    Let $S$ and $T$ be sets of $3$--independent vectors of length $M$ and $N$, respectively, where $|S|=k_1$ and $|T|=k_2$. Define $S' = \{{\bf u} \oplus {\bf v} : {\bf u} \in S, {\bf v} \in T \} $. Observe that $S'$ satisfies the conditions (i) and (ii) in Theorem \ref{thm:t_independent_FR_vectors} and $\vert S' \vert = k_1k_2$. We claim that $S'$ is also $3$--independent.

    Let ${\bf x}, {\bf y}, {\bf z} $ be three distinct elements in $ S'$. Then ${\bf x} = {\bf a} \oplus {\bf d}$, ${\bf y} = {\bf b} \oplus {\bf e}$, and $ {\bf z} = {\bf c} \oplus {\bf f} $, where ${\bf a}, {\bf b}, {\bf c} \in S$ and ${\bf d},{\bf e},{\bf f} \in T$. Without loss of generality, we have the following two cases to consider: \\
    \textbf{Case I}: ${\bf a}, {\bf b}, {\bf c}$ are all distinct. In this case ${\bf x},{\bf y},{\bf z}$ are linearly independent. \\
    \textbf{Case II}: ${\bf a}={\bf b} \neq {\bf c}$. Observe that ${\bf d} \neq {\bf e}$ in this case. Suppose that there exist $\alpha, \beta, \gamma \in \mathbb{F}_q$ such that: 
    \begin{equation*}
    \alpha {\bf x} + \beta {\bf y} + \gamma {\bf z} = {\bf 0}
    \end{equation*}
    Then we have:
    \begin{equation} \label{eq:t=3_abc}
        \alpha {\bf a} + \beta {\bf a} + \gamma {\bf c} = {\bf 0} \
    \end{equation}
    \begin{equation} \label{eq:t=3_def}
        \alpha {\bf d} + \beta {\bf e} + \gamma {\bf f} = {\bf 0} 
    \end{equation}
    From equation (\ref{eq:t=3_abc}),  $(\alpha + \beta){\bf a} = -\gamma {\bf c}$. Since ${\bf a}$ and ${\bf c}$ are linearly independent, $ \gamma = 0$ and $\alpha = - \beta$. But then equation (\ref{eq:t=3_def}) implies ${\bf d} = {\bf e} $, which is a contradiction, since ${\bf d}$ and ${\bf e}$ are two distinct elements of $3$--independent set $T$.
\end{proof}

The above results indicate that there is a close relationship between a set of $t$--independent vectors over $\mathbb{F}_q$ and $t$--orthogonal frequency rectangles. Therefore, we are interested in the maximum size of sets of linearly independent vectors corresponding to different sets of parameters. For particular values of $ N,q $ and $ t $, let Ind$_q(N,t) $ denote the maximum possible size of a set of $t$--independent vectors of length $N$ over a finite field $\mathbb{F}_q$. In the case of $q=2 $, we drop $q$ and simply write Ind$(N,t)$. As we know that any two vectors are linearly independent if and only if one is not a scalar multiple of the other, it is easy to see that Ind$_q(N,2) = (q^N - 1)/(q-1) $. Also for any $t \geq 2$, Ind$_q(N,t-1) \geq$ Ind$_q(N,t) $. The values Ind$_q(N,t) $ when $t \geq 3 $ are of interest to researchers in coding theory, combinatorics, and matroid theory \cite{ball2012sets, ball2012setsII, damelin2004number, damelin2007cardinality}. The following result in the case $q=2$ is taken from \cite{damelin2004number}.

\begin{theorem} \label{thm:Ind_basic}
	For $ q=2  $ the following formulae hold:
	
	\begin{enumerate}
		\item [\textup{(a)}] \begin{equation}
			\textup{Ind}(N,3)=2^{N-1}, \hspace{1cm} for \hspace{0.2cm} N \geq 3.
		\end{equation}
		\item [\textup{(b)}] \begin{equation}
			\textup{Ind}(N,N-r)=N+1, \hspace{1cm} for \hspace{0.2cm} N\geq3r+2, \ \ r\geq 0.
		\end{equation}
		\item [\textup{(c)}] \begin{equation}
			\textup{Ind}(N,N-r)=N+2, \hspace{0.6cm} for \ N=3r+i, \ i= 0,1, \ \ r \geq2.
		\end{equation}
	\end{enumerate}
\end{theorem}

Part (b) of the above theorem was later generalized  in \cite{damelin2007cardinality}.

\begin{theorem} \textup{\cite{damelin2007cardinality}} \label{thm:q_by_q+1_t_ind}
	Let $ 2 \leq t \leq N $. Then \textup{Ind}$_q(N,t)=N+1  $ if and only if
		\begin{equation*}
			 \frac{q}{q+1}(N+1)\leq t.
		\end{equation*}	 
\end{theorem}

The next two results discuss the upper bounds on Ind$_q(N,t) $ in the case when $t = N$.

\begin{theorem} \textup{\cite{ball2012sets}}
Let $q = p^h$, where $p$ is a prime. Then
\begin{enumerate}
		\item [\textup{(a)}] \begin{equation}
			\textup{Ind}_q(N,N) \leq q+1, \hspace{1cm} if \ N \leq p.
		\end{equation}
		\item [\textup{(b)}] \begin{equation}
			\textup{Ind}_q(N,N) \leq q+N-p, \hspace{1cm} if \ q \geq N \geq p+1 \geq 4.
		\end{equation}
\end{enumerate}
\end{theorem}

\begin{theorem} \textup{\cite{ball2012setsII}}
Let $q = p^h$, where $p$ is a prime, and let $N \leq 2p-2$. Then
\begin{enumerate}
		\item [\textup{(a)}] \begin{equation}
			\textup{Ind}_q(N,N) \leq q+2, \hspace{1cm} if \ \textup{$q$ is even and} \ N=3 \ or \ N = q-1.
		\end{equation}
		\item [\textup{(b)}] \begin{equation}
			\textup{Ind}_q(N,N) \leq q+1, \hspace{1cm} otherwise.
		\end{equation}
\end{enumerate}
\end{theorem}
\noindent The following values of Ind$_q(N,t)$ for $ q = 3$ are listed in \cite{tassa2009proper}:

\begin{itemize}
\item Ind$_3(5,3) = 20$ 
\item Ind$_3(5,4) = 11$
\item Ind$_3(6,3) = 56$
\item Ind$_3(6,4) = 13$
\item Ind$_3(6,5) = 13$
\end{itemize}

Next, we discuss the connection between $t$--independent vectors and the field of coding theory. In the rest of this section, we restrict ourselves to the binary case. 

A \emph{linear code} $C$ of length $n$, dimension $k$ is a subspace of $(\mathbb{F}_q)^n$ of dimension $k$. Let $c_1, c_2 \in C$. The \emph{hamming distance} $d(c_1,c_2)$ between the codewords $c_1$ and $c_2$ is the number of positions at which they differ.  The minimum hamming distance $d$ is defined as follows. 
$$ d = \min\{d(c_1, c_2): c_1, c_2 \in C\}$$
A code with length $n$, dimension $k$, and minimum hamming distance $d$ is called an $[n,k,d]$-code. 
The following result shows the relationship between a set of $t$--independent vectors and a linear $[n,k,d]$-code.

\begin{lemma} \label{lem:codes_ind_vectors}
\textup{\cite{Colbourn:2006:HCD:1202540}}
There exists a linear  $[n,k,d]$-code if and only if there exists a $(n-k)\times n$ matrix $H$ such that any $d-1$ columns of $H$ are linearly independent, but there exists a set of $d$ columns which are not linearly independent.     
\end{lemma}    

\begin{corollary} \label{cor:code_implies_Ind}
If there exists a linear $[n,k,d]$-code, then $\textup{Ind}(n-k,d-1)\geq n$.     
\end{corollary}

\begin{lemma}
If $\textup{Ind}(n-k,d-1)\geq n$, then there exists a linear $[n,k,d']$-code for some $d'\geq d$.  
\end{lemma}

\begin{proof}
Let $\textup{Ind}(n-k,d-1)\geq n$. Then there exists a set of $n$ vectors, each of length $n-k$ such that any subset of $d-1$ vectors is linearly independent. Let $H$ be the $(n-k)\times n$ matrix whose columns are these vectors. Let $d'$ be the smallest integer such that there exists a set of $d'$ linearly dependent columns in $H$. Clearly $d'\geq d$. Moreover, any subset of $d'-1$ columns of $H$ is linearly independent. Thus by Lemma \ref{lem:codes_ind_vectors}, there exists a
linear $[n,k, d']$-code. 
\end{proof}

The following corollary is the contrapositive of the previous lemma. 

\begin{corollary}
Suppose that for all $d'\geq d$, there does not exist a linear $[n,k,d']$-code. Then  $\textup{Ind}(n-k,d-1)< n$.    
\end{corollary}

\begin{corollary} \label{cor:codes_t_independent}
Let $D$ be the maximum value such that a linear $[n,k, D]$-code exists. 
 Then  $\textup{Ind}(n-k,D)< n$.    
\end{corollary}

Now we present some known values and bounds on Ind$(N,t)$ for $N \leq 16$ in Table  \ref{tab:Ind_table_codes}.
Since  Ind$(N,2) = 2^N - 1$ and Ind$(N,3) = 2^{N-1} $ (from Theorem \ref{thm:Ind_basic}(a)), we restrict ourselves to  the values  $t \geq 4$.
Let $D(n,k)$ be the maximum value $d$ such that a linear $[n,k,d]$-code exists. 

Most of the results in Table \ref{tab:Ind_table_codes} are using an online repository for linear codes \cite{markus2009codestable} in conjunction with Corollary \ref{cor:code_implies_Ind} and  Corollary \ref{cor:codes_t_independent}. Some are using the above results and \cite{tassa2009proper}. The reasoning for each case is provided in the description. Also, the notation  ``$a$ -- $b$'' in column Ind$(N,t)$ implies $a \leq $ Ind$(N,t) \leq b$.

Let us  compute Ind$(9,4)$ as an example. From \cite{markus2009codestable}, a $[23, 14, 5]$-code exists; by Corollary \ref{cor:code_implies_Ind} this implies Ind$(9,4) \geq 23$. Also from \cite{markus2009codestable},  $D(24, 15) = 4$. Thus  by Corollary \ref{cor:codes_t_independent}, Ind$(9,4) < 24$. Consequently, we have Ind$(9,4) = 23$.  
\renewcommand{\arraystretch}{1.2}
\begin{longtable}[c]{ccc l} 
\kill
\caption[]{Values for Ind$(N,t)$ for $N \leq 16$\label{tab:Ind_table_codes}}\\
\hline\hline
\textbf{$N$} & \hspace{0.5cm} $ \textbf{\textit{t}} $ \hspace{0.5cm} & \textbf{\textup{Ind}$ (N,t) $} &  \textbf{ \hspace{0.6cm} Description}\\
\hline\hline
\endfirsthead
\caption[]{(continued)}\\
\hline\hline
\textbf{$N$} & $ \textbf{\textit{t}} $ & \textbf{\textup{Ind}$ (N,t) $} &  \textbf{ \hspace{0.6cm} Description}\\
\hline \hline
\endhead
        $ 5 $   & $ 4 $   & $ 6 $  & \cite{tassa2009proper} \\*
                & $ 5 $   & $ 6 $  &  Theorem \ref{thm:q_by_q+1_t_ind}\\\midrule
		$ 6 $   & $ 4 $   & $ 8 $& \cite{tassa2009proper} \\*
		 	& $5,6$   & $ 7 $&  Theorem \ref{thm:q_by_q+1_t_ind} \\* \midrule 
		$ 7 $   & $ 4 $   & $ 11 $& \cite{tassa2009proper} \\*
			  & $ 5 $   & $ 9 $&  \cite{tassa2009proper} \\*
			  & $6,7$   & $ 8 $&  Theorem \ref{thm:q_by_q+1_t_ind} \\ \midrule 
		$ 8 $   & $ 4 $   & $ 17 $& \cite{tassa2009proper} \\*
			  & $ 5 $   & 12 &  A $ [12, 4, 6] $-code exists and $D(13,5) = 5$ (\cite{markus2009codestable}).   \\*
			  & $ 6 $   & 9 & Theorem \ref{thm:Ind_basic}(b)  \\*
                & $7,8$  & 9 & Theorem \ref{thm:q_by_q+1_t_ind}  \\ \midrule
		$ 9 $   & $ 4 $   & 23   & A $ [23,14,5] $-code exists and $D(24, 15)= 4 $ (\cite{markus2009codestable}).  \\*
			  & $ 5 $   & 18 & A $ [18,9,6] $-code exists and $D(19,10)= 5$ (\cite{markus2009codestable}).   \\*
			  & $ 6 $   & 11 & A $ [11,2,7] $-code exists and $D(12,3)= 6$ (\cite{markus2009codestable}). \\*
			  & $7,8,9$   & 10 & Theorem \ref{thm:q_by_q+1_t_ind} \\ \midrule
     $ 10 $     & $ 4 $   & 33  & A $ [33,23,5] $-code exists and $D(34,24)= 4$ (\cite{markus2009codestable}). \\*
			& $ 5 $   & 24  & A $ [24,14,6] $-code exists and $D(25,15) = 5$ (\cite{markus2009codestable}).   \\*
			  & $ 6 $   & 15 & A $ [15,5,7] $-code exists and $D(16,6) = 6$ (\cite{markus2009codestable}). \\*
			  & $ 7 $   & 12 & Theorem \ref{thm:Ind_basic}(c) \\*
			  & $8,9,10$  & 11 & Theorem \ref{thm:q_by_q+1_t_ind} \\ \midrule
		$ 11 $ & $ 4 $   & 47 -- 57 & A $ [ 47,36,5] $-code exists and $D(58,47)= 4$ (\cite{markus2009codestable}).  \\*
			   & $ 5 $   & 34  & A $ [34,23,6] $-code exists and $D(35,24)= 5$ (\cite{markus2009codestable}).   \\*
			   & $ 6 $   & 23 & A $ [23,12,7] $-code exists and $D(24,13)= 6$ (\cite{markus2009codestable}). \\*
			   & $ 7 $   & 16 & A $ [16,5,8] $-code exists and $D(17,6)= 7$ (\cite{markus2009codestable}). \\*
			   & $ 8 $   & 12 & Theorem \ref{thm:Ind_basic}(b) \\*
                  & $9,10,11$   & 12 & Theorem \ref{thm:q_by_q+1_t_ind} \\ \midrule
		$ 12 $ & $ 4 $   & 65 -- 88  & A $ [65,53,5] $-code exists and $D(89,77)= 4$ (\cite{markus2009codestable}).   \\*
			   & $ 5 $   & 48 -- 58  & A $ [48,36,6] $-code exists and $D(59,47)= 5$ (\cite{markus2009codestable}).  \\*
			   & $ 6 $   & 24 & A $ [24,12,8] $-code exists and $D(25,13)= 6$ (\cite{markus2009codestable}). \\*
			   & $ 7 $   & 24 & A $ [24,12,8] $-code exists and $D(25,13)= 6$ (\cite{markus2009codestable}). \\*
			   & $ 8 $   & 14 & Theorem \ref{thm:Ind_basic}(c) \\*
			   & $9, \dots, 12$   & 13 & Theorem \ref{thm:q_by_q+1_t_ind} \\ \midrule
      $ 13 $ & $ 4 $   & 81 -- 124  & A $ [81,68,5] $-code exists and $D(125,112)= 4$ (\cite{markus2009codestable}).   \\*
			   & $ 5 $   & 66 -- 89  & A $ [66,53,6] $-code exists and $D(90,77)= 5$ (\cite{markus2009codestable}). \\*
			   & $ 6 $   & 27  & A $ [27,14,7] $-code exists and $D(28,15)= 6$ (\cite{markus2009codestable}). \\*
			   & $ 7 $   & 25 & A $ [25,12,8] $-code exists and $D(26,13)= 7$ (\cite{markus2009codestable}). \\*
			   & $ 8 $   & 15 & A $ [15,2,10] $-code exists and $D(16,3)= 8$ (\cite{markus2009codestable}).  \\*
			   & $ 9 $   & 15 & A $ [15,2,10] $ Theorem \ref{thm:Ind_basic}(c) \\*
			   & $10, \dots, 13$  & 14 & A $ [15,2,10] $ Theorem \ref{thm:q_by_q+1_t_ind} \\ \midrule
      $ 14 $ & $ 4 $   & 128 -- 178 & A $ [128,114,5] $-code exists and $D(179,165)= 4$ (\cite{markus2009codestable}).   \\*
			& $ 5 $   & 82 -- 125 & A $ [82,68,6] $-code exists and $D(126,112)= 5$ (\cite{markus2009codestable}).  \\*
			& $ 6 $   & 31 -- 40  & A $ [31,17,7] $-code exists and $D(41,27)= 6$ (\cite{markus2009codestable}). \\*
			& $ 7 $   & 28  & A $ [28,14,8] $-code exists and $D(29,15)= 7$ (\cite{markus2009codestable}).\\*
			& $ 8 $   & 17 & A $ [17,3,9] $-code exists and $D(18,4)= 8$ (\cite{markus2009codestable}). \\*
			& $ 9 $   & 16 & A $ [16,2,10] $-code exists and $D(17,3)= 9$ (\cite{markus2009codestable}). \\*
			& $10, \dots, 14$  & 15 & Theorem \ref{thm:Ind_basic}(b) \\ \midrule
	 $ 15 $ & $ 4 $   & 151 -- 253 & A $ [151,136,5] $-code exists and $D(254,239)= 4$ (\cite{markus2009codestable}). \\*
			& $ 5 $   & 129 -- 179 & A $ [129,114,6] $-code exists and $D(180,165)= 5$ (\cite{markus2009codestable}). \\*
			& $ 6 $   & 37 -- 53 & A $ [37,22,7] $-code exists and $D(54,39)= 6$ (\cite{markus2009codestable}). \\*
			& $ 7 $   & 32 -- 41 & A $ [32,17,8] $-code exists and $D(42,27)= 7$ (\cite{markus2009codestable}). \\*
			& $ 8 $   & 20 & A $ [20,5,9] $-code exists and $D(21,6)= 8$ (\cite{markus2009codestable}). \\*
			& $ 9 $   & 18 & A $ [18,3,10] $-code exists and $D(19,4)= 9$ (\cite{markus2009codestable}).  \\*
			& $ 10 $  & 17 & Theorem \ref{thm:Ind_basic}(c) \\*
			& $11, \dots, 15$ & 16 & Theorem \ref{thm:q_by_q+1_t_ind} \\\midrule
	 $ 16 $ & $ 4 $   & $ \geq $ 256 & A $ [256,240,5] $-code exists.   \\*
			& $ 5 $   & 152 -- 254 & A $ [152,136,6] $-code exists and $D(255, 239)= 5$ (\cite{markus2009codestable}). \\*
			& $ 6 $   & 47 -- 69 & A $ [47,31,7] $-code exists and $D(70,54)= 6$ (\cite{markus2009codestable}). \\*
			& $ 7 $   & 38 -- 54 & A $ [38,22,8] $-code exists and $D(55,39)= 7$ (\cite{markus2009codestable}). \\*
			& $ 8 $   & 23 & A $ [23,7,9] $-code exists and $D(24,8)= 8$ (\cite{markus2009codestable}). \\*
			& $ 9 $   & 21 & A $ [21,5,10] $-code exists and $D(22,6)= 9$ (\cite{markus2009codestable}).  \\*
			& $ 10 $  & 18 & A $ [18,2,12] $-code exists and $D(19,3)= 10$ (\cite{markus2009codestable}). \\* 
			& $ 11 $  & 18 & Theorem \ref{thm:Ind_basic}(c) \\*
			& $12, \dots, 16$  & 17 & Theorem \ref{thm:q_by_q+1_t_ind} \\ \midrule
\end{longtable}

If we know the value Ind$_q(N,t)$ then by using Corollary \ref{cor:t_indpendent_MOFR} we can construct a set of $k = $ Ind$_q(N,t)$ $t$--orthogonal MOFR$(q^N, q^M; q)$, where $M \geq N$. Though it provides a lower bound that is, in general, not close to the actual upper bound. For example, consider the case when $q = 2, t=3$ and $N =M= 2$. In this case, Ind$(2,3)$ does not exist and hence does not provide a  lower bound. However, by Theorem \ref{thm:Ind_basic}(a) we know Ind$(4,3) = 8$ and one such set is the set $O$ (given below) of vectors of odd weight in $(\mathbb{F}_2)^4$.
$$ O = \{1000, 0100, 0010,0001, 1110, 1101, 1011, 0111\}.$$

Observe that only 4 elements of $O$ satisfy the conditions of Theorem \ref{thm:t_independent_FR_vectors}. Thus we can construct a set of 4--MOFR$(4,4;2)$ that is 3--orthogonal by using these elements in $O$. On the other hand, by inspection, we have found the following set:
$$ W = \{1010, 1001, 1101, 0101, 1110, 0110, 0001, 0010\}$$
which is $3$--independent and has 6 elements that satisfy the conditions of Theorem \ref{thm:t_independent_FR_vectors}. In fact, the first 6 elements listed in $W$ were used to construct 3--orthogonal 6-MOFR$(4,4,2)$ in Example \ref{exp:6_MOFR(4,4,2)}.  

This motivates us to propose the following problem.

\begin{problem} Let $S  \subseteq (\mathbb{F}_q)^{M+N}$. For each set of admissible parameters $t, M, N$, and $q$, determine the maximum size of $S$ such that:
\begin{enumerate}
    \item[\textup{(i)}] $S$ is $t$--independent.
    \item[\textup{(ii)}] For each $ {\bf v} = (v_1, \dots, v_{M+N}) \in S$,   $(v_1, \dots, v_M) \neq (0, \dots, 0) $ and $(v_{M+1}, \dots, v_{M+N}) \neq (0, \dots, 0) $.
\end{enumerate}
\end{problem}

\section{$p-1$ binary MOFS of size $2p$} \label{sec:(p-1)_binary_MOFS}

In this section, our aim is to describe a method to construct a set of $p-1$ mutually orthogonal frequency squares of order $2p$, where $p$ is an odd prime. The construction starts by generating a set of $p-1$ frequency squares which are almost orthogonal. Then we make some small changes in each frequency square in order to make the set orthogonal. Here we set out some notations that we use frequently in this section.

Let $[n] = \{0,1, \dots, n-1\}$, where  $n$ is an integer. Let $H$ be an $m \times n$ array. The rows and columns of $H$ are indexed using $[m]$ and $[n]$. The entry in the intersection of row $i$ and column $j$ of the array $H$ is denoted by $h(i,j)$. As previously, if $H$ is a binary array then $\overline{H}$ is the array obtained from $H$ by interchanging zeroes and ones.

Let $A$ and $B$ be two $m \times n$ binary arrays. We use the notation $\vert AB \vert_{(x,y)},$ where $ x,y \in \{0,1\} $, to denote the total number of ordered pairs $(i,j)$ such that $a(i,j) = x$ and $b(i,j) = y$, that is the total number of ordered pairs of type $(x,y)$ obtained when $A$ and $B$ are superimposed. The term \emph{orthogonality} between the arrays $A$ and $B$ means a sequence that contains the numbers $\vert AB \vert_{(x,y)} $ for all $ x,y \in \{0,1\} $. However, the term \emph{orthogonal} has its usual meaning, i.e., $A$ and $B$ are said to be orthogonal if each type of ordered pair appears the same number of times upon superimposition.

\sloppy Recall that if ${\bf c} = (c_0, c_1, \dots , c_r)$ is a vector then by a cyclic shift of ${\bf c}$ we mean the vector ${\bf c}' = (c_r, c_0, c_1, \dots, c_{r-1})$. Let $p$ be an odd prime. Let ${\bf v} \in (\mathbb{F}_2)^p$ be the vector ${\bf v} = (1,1, \dots, 1, 0, 0, \dots, 0)$ of weight $(p+1)/2$. Let ${\bf v}_i$ denote the vector obtained from ${\bf v}$ by performing $i$ cyclic shifts. Throughout this section $\Omega = \{1, \dots, p-1\}$ and $K = \{1, \dots, \frac{p-1}{2}\}$. We include here some observations related to the vectors ${\bf v}_i$.

\begin{lemma} \label{lem:number_of_pairs}
Let $z \in [(p+1)/2]$. Upon superimposing ${\bf v}_i {\bf v}_j$, the following are equivalent:
\begin{enumerate}
    \item[\textup{(a)}] $\vert {\bf v}_i{\bf v}_j \vert_{(1,0)}=  z$.
    \item[\textup{(b)}] $\vert {\bf v}_i{\bf v}_j \vert_{(0,1)}= z$.
    \item[\textup{(c)}] $\vert {\bf v}_i{\bf v}_j \vert_{(0,0)}= \frac{p-1}{2} - z$.
    \item[\textup{(d)}] $\vert {\bf v}_i{\bf v}_j \vert_{(1,1)}= \frac{p+1}{2} - z$.
\end{enumerate}
\end{lemma}
\begin{proof}
Each vector contains exactly $(p-1)/2$ zeroes.
\end{proof}

\begin{corollary} \label{cor:commutative_vi_vj}
Superimposing ${\bf v}_i$ and $ {\bf v}_j$ in any order yields the same number of ordered pairs of each type, i.e., $\vert {\bf v}_i{\bf v}_j \vert_{(x,y)} = \vert {\bf v}_j{\bf v}_i \vert_{(x,y)} $.
\end{corollary}

\begin{lemma} \label{lem:v0_vi_in_terms_of_i}
For $i \in  [(p+1)/2]$,
\begin{enumerate}
    \item[\textup{(a)}] $\vert {\bf v}_0{\bf v}_i \vert_{(1,0)}=  \vert {\bf v}_0{\bf v}_i \vert_{(0,1)}  = i$.
    \item[\textup{(b)}] $\vert {\bf v}_0{\bf v}_i \vert_{(0,0)}= \frac{p-1}{2} - i$.
    \item[\textup{(c)}] $\vert {\bf v}_0{\bf v}_i \vert_{(1,1)}= \frac{p+1}{2} - i$.
\end{enumerate}
\end{lemma}

\begin{lemma} \label{lem:v0_vi}
For any $i,j \in [p]$ and $x,y \in \{0,1\}$, we have the following:
\begin{enumerate}
    \item[\textup{(a)}] $\vert {\bf v}_0{\bf v}_i \vert_{(x,y)} = \vert {\bf v}_0{\bf v}_{p-i} \vert_{(x,y)} $.
    \item[\textup{(b)}] $\vert {\bf v}_i{\bf v}_j \vert_{(x,y)} = \vert {\bf v}_0{\bf v}_r \vert_{(x,y)} $, where $r \in [p]$, $ r \equiv j - i \pmod{p}$.
\end{enumerate}
\end{lemma}
\begin{proof}
    This follows from Corollary \ref{cor:commutative_vi_vj} and the observation $\vert {\bf v}_i{\bf v}_j \vert_{(x,y)} = \vert {\bf v}_{i+1}{\bf v}_{j+1} \vert_{(x,y)}  $.
\end{proof}

Now we define a set of $p-1$ arrays each of which consists of a different permutation of the vectors ${\bf v}_i$. These arrays will be our primary building blocks in defining our frequency squares. Formally, for each $\alpha \in \Omega$, let $A_\alpha$ be a $p \times p$ binary array such that for each $i \in [p]$ its row $i$ is ${\bf v}_r$, where $r \in [p]$ is the unique solution to $r \equiv \alpha i \pmod{p} $. The next lemma records the total number of order pairs of each type obtained when these arrays are superimposed.

\begin{lemma} \label{lem:order_pairs_in_A_alpha_A_beta}
Let $\alpha, \beta \in \Omega$ and $ \alpha \neq \beta$. Then:
\begin{enumerate}
    \item[\textup{(a)}]$\vert A_{\alpha} A_{\beta} \vert_{(1,0)} = \vert A_{\alpha} A_{\beta} \vert_{(0,1)} = (p^2 - 1)/4 $.
    \item[\textup{(b)}] $ \vert A_{\alpha} A_{\beta} \vert_{(0,0)} = (p - 1)^2/4 $.
    \item[\textup{(c)}] $ \vert A_{\alpha} A_{\beta} \vert_{(1,1)} = (p + 1)^2/4$.
\end{enumerate}
\end{lemma}

\begin{proof}
When we superimpose $A_{\alpha}A_{\beta}$ the corresponding rows are ${\bf v}_{r_1(i)}{\bf v}_{r_2(i)}$ where $r_1(i),r_2(i) \in [p]$ and  $r_1(i) \equiv \alpha i \pmod{p}$, $r_2(i) \equiv \beta i \pmod{p}$ for each $i \in [p]$. Thus,
\begin{align*}
 \vert A_{\alpha} A_{\beta} \vert_{(x,y)} &= \sum_{i=0}^{p-1} \vert {\bf v}_{r_1(i)} {\bf v}_{r_2(i)} \vert_{(x,y)} \hspace{1cm} r_1(i),r_2(i) \in [p],  \hspace{0.5cm} r_1(i) \equiv \alpha i \pmod{p}, \hspace{0.5cm} r_2(i) \equiv \beta i \pmod{p} \\
 &= \sum_{i=0}^{p-1} \vert {\bf v}_{0} {\bf v}_{r(i)} \vert_{(x,y)} \hspace{1cm} r(i) \in [p], \hspace{0.5cm}  r(i) \equiv (\beta - \alpha)i \pmod{p} \hspace{0.7cm} \textup{(by Lemma \ref{lem:v0_vi}(b))}.
\end{align*}
Since $ \alpha \not\equiv \beta \pmod{p}$,  $\{ r(i): i \in [p] \}$  forms a complete set of residues modulo $p$. Thus by using Lemma \ref{lem:v0_vi}(a) we have:
\begin{equation} \label{eq:A_alphaA_beta}
\vert A_{\alpha} A_{\beta} \vert_{(x,y)} = \sum_{i=0}^{p-1} \vert {\bf v}_{0} {\bf v}_{i} \vert_{(x,y)} = \vert {\bf v}_{0} {\bf v}_{0} \vert_{(x,y)} + 2\sum_{i=1}^{(p-1)/2} \vert {\bf v}_{0} {\bf v}_{i} \vert_{(x,y)}
\end{equation}
Now by using the values of $\vert {\bf v}_{0} {\bf v}_{i} \vert_{(x,y)}$ from Lemma \ref{lem:v0_vi_in_terms_of_i} we get:
$$ \vert A_{\alpha} A_{\beta} \vert_{(1,0)} = \vert A_{\alpha} A_{\beta} \vert_{(0,1)} = 0 + 2\sum_{i=1}^{(p-1)/2}{i} = (p^2 -1)/4,$$
$$ \vert A_{\alpha} A_{\beta} \vert_{(0,0)} = \frac{p-1}{2} + 2\sum_{i=1}^{(p-1)/2}{(\frac{p-1}{2} - i )} = \frac{p(p-1)}{2} - 2\sum_{i=1}^{(p-1)/2}{i} = \frac{(p-1)^2}{4}.$$
Now (c) follows since:
$$ \sum_{x,y \in \{0,1\}} \vert A_{\alpha} A_{\beta} \vert_{(x,y)} = p^2.$$
\end{proof}

\begin{corollary} \label{cor:order_pairs_in_A_alpha_A_beta}
Let $\alpha, \beta \in \Omega$ and $ \alpha \neq \beta$. Then:
\begin{enumerate}
    \item[\textup{(a)}]$\vert \overline{A}_{\alpha} \overline{A}_{\beta} \vert_{(1,0)} = \vert \overline{A}_{\alpha} \overline{A}_{\beta} \vert_{(0,1)} = (p^2 - 1)/4 $.
    \item[\textup{(b)}] $ \vert \overline{A}_{\alpha} \overline{A}_{\beta} \vert_{(0,0)} = (p + 1)^2/4 $.
    \item[\textup{(c)}] $ \vert \overline{A}_{\alpha} \overline{A}_{\beta} \vert_{(1,1)} = (p - 1)^2/4$.
\end{enumerate}
\end{corollary}

Now let us construct a set of $p-1$ frequency squares. Corresponding to each $\alpha \in \Omega$ we construct a binary frequency square $L_\alpha$ of order $2p$ as follows. 
\begin{equation} \label{eq:L_alpha_1st}
L_\alpha = \left[
\renewcommand{\arraystretch}{1.2}
\begin{array}{c|c}
A_\alpha & \overline{A_\alpha} \\ \hline
\overline{A_\alpha} & A_\alpha \\
\end{array} \right]
\end{equation}

For $\alpha \neq \beta$, by Lemma \ref{lem:order_pairs_in_A_alpha_A_beta} and Corollary \ref{cor:order_pairs_in_A_alpha_A_beta} the total number of ordered pairs $(x,y)$ when $L_{\alpha}$ and $L_{\beta}$ superimposed are as follows:
\begin{equation} \label{eq:L_alpha_L_beta_1st}
\begin{split}
\vert L_{\alpha} L_{\beta} \vert_{(1,0)} = \vert L_{\alpha} L_{\beta} \vert_{(0,1)} = p^2 - 1 \\
\vert L_{\alpha} L_{\beta} \vert_{(0,0)} = \vert L_{\alpha} L_{\beta} \vert_{(1,1)} = p^2 + 1
\end{split}
\end{equation}
Thus we have a set of $p-1$ frequency squares that are almost orthogonal. Now we will make some small changes in order to make these squares pairwise orthogonal. Specifically, in the first and the fourth quadrant, we will flip some entries in a way that some of these arrays will become orthogonal to each other without disrupting the orthogonality between other pairs.

We introduce some more terminology here that we will use often in the rest of this section. 
Let $H_1, \dots, H_n$ be a set of $n$ binary arrays. Let $h_\alpha(i,j)$ denote the entry in the intersection of row $i$ and column $j$ of  $H_\alpha $, where $ i, j \in [n] $. Let $s_\alpha$ be a $2 \times 2$ sub-array of $H_\alpha$. Then the arrays $s_1, \dots, s_n$  are said to be \emph{coincident} if for each $\alpha \in \{1, \dots, n\}$, $s_\alpha$ is of the form:
$$ s_\alpha = \left( \begin{array}{cc} h_\alpha(i,j) & h_\alpha(i,j') \\ h_\alpha(i',j) & h_\alpha(i',j') \end{array} \right), \hspace{1cm} \textup{ for some fixed } i,j,i',j' \in [n] $$
that is, each $s_\alpha$ is a sub-square of $H_\alpha$ and all $s_\alpha$ corresponds to the same positions within $H_\alpha$.

The next lemma describes the effect on the orthogonality between any two binary arrays when we flip the entries of a $2 \times 2$ sub-array of one of them.

\begin{lemma} \label{lem:trades}
Let $A$ and $B$ be two binary arrays of size $m \times n$, where $m,n \geq 2$. Let $s_1$, $s_2$ be two $2 \times 2$ coincident sub-arrays of $A$ and $B$, respectively. Let $B'$ be the array obtained from $B$ in which $s_2$ is replaced by $\overline{s_2}$, where $s_2 = \bigl( \begin{smallmatrix}1 & 0 \\ 0 & 1 \end{smallmatrix} \bigr)$. 
\begin{enumerate}
    \item If $ s_1 = \bigl( \begin{smallmatrix} 1 & 0 \\ 1 & 1 \end{smallmatrix} \bigr)$, then:
    \begin{enumerate}
        \item[\textup{(a)}] $\vert AB'\vert_{(x,y)} = \vert B'A\vert_{(x,y)} = \vert AB\vert_{(x,y)} + 1$, where $x,y \in \{0,1\}$ and $x \neq y$.
        \item[\textup{(b)}] $\vert AB'\vert_{(x,x)} = \vert B'A\vert_{(x,x)}  = \vert AB\vert_{(a,a)} - 1$, where $x \in \{0,1\}$.
    \end{enumerate}
    \item If $ s_1 = \bigl( \begin{smallmatrix} 1 & 0 \\ 1 & 0 \end{smallmatrix} \bigr)$, then $\vert AB'\vert_{(x,y)} = \vert B'A \vert_{(x,y)} = \vert AB\vert_{(x,y)} $, for all $x,y \in \{0,1\}$.
    \item If $B$ is a binary frequency square then $B'$ is a binary frequency square.
\end{enumerate}
\end{lemma}

In what follows it will be useful to give an explicit formula for the entry of each cell in $A_\alpha$.

\begin{lemma} \label{lem:entry_i_j}
Let $a_{\alpha}(i,j)$ denote the entry in the intersection of row $i$ and column $j$ of $A_{\alpha}$, where $i,j \in [p]$. Then 
$$ a_{\alpha}(i,j) = \begin{cases}
 0 & \textup{if } \hspace{0.5cm} (1-p)/2 \leq j-r \leq -1  \hspace{0.5cm} \textup{or} \hspace{0.5cm}  (p+1)/2 \leq j-r \leq p-1 \\
 1 & otherwise,
 \end{cases} 
 $$
 where $r \equiv \alpha i \pmod{p} \hspace{0.2cm} \textup{and} \hspace{0.2cm} 0 \leq r \leq p-1    $.
\end{lemma}
Now we describe different sets of coincident sub-arrays of $A_\alpha$ that we will use later on to alter the orthogonality between the arrays. Formally, let $a_{\alpha}(i,j)$ denote the entry in the intersection of row $i$ and column $j$ of $A_{\alpha}$, where $i,j \in [p]$. For each $h \in K$ and $\alpha \in \Omega $ we define:
\begin{equation} \label{eq:intercalates}
s_{(\alpha, h)} = \left(
\renewcommand{\arraystretch}{1.5}
\begin{array}{cc}
    a_{\alpha}(0,h) & a_{\alpha}(0,\frac{p-1}{2} + h) \\
    a_{\alpha}(1,h) & a_{\alpha}(1,\frac{p-1}{2} + h)
\end{array} \right)
\end{equation}
Then for each $h \in K$ the set $\{s_{(\alpha, h)}: \alpha \in \Omega \}$ forms a set of coincident sub-arrays of $A_\alpha$. The configuration of each sub-array is given in the following lemma.

\begin{lemma} \label{lem:intercalates}
For each $h,k \in K$ and $\beta \in \Omega \setminus \{h, h+1, \dots, h+\frac{p-1}{2}\}$ we have:
$$
s_{(h, h)} = \left( 
\begin{array}{cc}
    1 & 0 \\
    1 & 1 
\end{array} \right), \hspace{0.2cm}
s_{(h+k, h)} = \left( 
\begin{array}{cc}
    1 & 0 \\
    0 & 1 
\end{array} \right),  \hspace{0.2cm} 
 s_{(\beta, h)} = \left( 
\begin{array}{cc}
    1 & 0 \\
    1 & 0 
\end{array} \right).$$
\end{lemma}

\begin{proof}
This follows by Lemma \ref{lem:entry_i_j}.
\end{proof}

We next verify that the coincident arrays that we later use do not overlap.  
The following lemma follows from the definition (\ref{eq:intercalates}) of $s_{(\alpha, h)}$.

\begin{lemma} \label{lem:s_alphas_dont_overlap}
Let $h, h' \in K$ such that $h \neq h'$. Then 
for each $\alpha\in \Omega$,
 the set of cells in 
 $s_{(\alpha, h)}$
 is disjoint from the set of cells in $s_{(\alpha, h')}$. 
\end{lemma}

Consider the set $\{A_\alpha : \alpha \in \Omega \}$ in the first quadrant of the arrays $L_\alpha$ (given in \ref{eq:L_alpha_1st}). By Lemma \ref{lem:intercalates} and Lemma \ref{lem:trades}, for each  $h \in K$, if we replace $s_{(h+k, h)}$ with $\overline{s}_{(h+k,h)}$  
for all $k\in K$, 
the array $L_h$ becomes orthogonal to each array in the set $\{L_{h+1}, L_{h+2}, \dots, L_{h+(p-1)/2}\}$, while the orthogonality between rest of the arrays remains unchanged. However, this does not cover the whole spectrum of the pairs which are subsets of $\{L_\alpha: \alpha \in \Omega\}$. Our next step is to apply similar transformations to the arrays $A_\alpha$ in the fourth quadrant of $L_\alpha$. To ensure that we are covering the whole spectrum without repetitions, we will use the analogy of a complete graph. That is we want to establish a one-to-one correspondence between the edges of a complete graph with $p-1$ vertices such that each vertex represents a unique frequency square and an edge between two vertices implies that the two arrays are orthogonal.
We start with the following result from graph theory.
\begin{lemma}
Let $G_{p-1}$ be the complete graph with vertex set $ V = \{\infty \} \cup \{0,1, \dots, p-3\}$. Let $S_i$ be the star with edge-set $\{\{i,\infty\}, \{i, i +1\}, \{i, i+2\}, \dots, \{i, i + \frac{p-3}{2} \} \}$ (working modulo $(p-2)$ with residues in $\{0,1, \dots, p-3\}$). Then $\{S_1, S_2, \dots, S_{p-2} \}$ is a partition of the edge set of $G_{p-1}$.
\end{lemma}

Next, we relabel the vertices of the graph $G_{p-1}$ by the mapping $f: V \rightarrow \Omega$ defined as follows:
$$ f(z) = \begin{cases}
 (p+1)/2 & \textup{if } \hspace{0.5cm} z = \infty \\
 p-1 & \textup{if } \hspace{0.5cm} z = 0 \\
 z & \textup{if } \hspace{0.5cm} 1 \leq z \leq (p-1)/2 \\
 z + 1 & \textup{if } \hspace{0.5cm} (p+1)/2 \leq z \leq p-3.
 \end{cases} 
 $$
The relabelling of vertices by using $f$ is shown in the figure given below.

\begin{table}[H]
\begin{minipage}[h]{.42\textwidth}
\centering
\begin{tikzpicture}  
  [scale=1,auto=center,every node/.style={circle,fill=blue!10, minimum size=10mm}, roundnode/.style={circle, fill=none } ] 

  \node (a1) at (3,-6) {1};  
  \node (a2) at (4.5,-5.5)  {2};
  \node (a3) at (5.7,-4.2)  {3};  
  \node (a4) at (5.9,-2.5)  {4};
  \node (ap-3) at (5.2,-0.9)  {$\frac{p-3}{2}$};
  \node (ap-1) at (3.85,-0.15)  {$\frac{p-1}{2}$};
  \node (ap+3)  at (2.15,-0.15)  {$\frac{p+1}{2}$} ;
  \node (ap+5) at (0.8,-0.9)  {$\frac{p+3}{2}$};
  \node (ap+7) at (0.1,-2.5)  {$\frac{p+5}{2}$};
  \node (p-2) at (0.3,-4.2)  {\small{$p-3$}};
  \node (p-1) at (1.5,-5.5)  {\small{$0$}};
  \node (a12) at (3,-3)  {\small{$\infty$}};
    
  
  \path (a4) -- node[auto=false, fill=none, sloped]{\ldots} (ap-3);
  \path (ap+7) -- node[auto=false, fill=none, sloped]{\ldots} (p-2);

\end{tikzpicture}
\caption*{$V$}
\end{minipage}
\begin{minipage}[h]{.15\textwidth}
\centering
\begin{tikzpicture}  
  [scale=1,auto=center,every node/.style={circle,fill=blue!10, minimum size=10mm}, roundnode/.style={circle, fill=none } ] 

   \node[auto=false, fill=none] (A)  at (1.15,-0.15)  {} ;
   \node[auto=false, fill=none] (B)  at (3.15,-0.15)  {} ;

   \draw[->] (A) -- (B) node[fill=none, auto=false, midway, above] {$f$};

\end{tikzpicture}
\end{minipage}  
\begin{minipage}[h]{.42\textwidth}
\centering
\begin{tikzpicture}  
  [scale=1,auto=center,every node/.style={circle,fill=blue!10, minimum size=10mm}, roundnode/.style={circle, fill=none } ] 

  \node (a1) at (3,-6) {1};  
  \node (a2) at (4.5,-5.5)  {2};
  \node (a3) at (5.7,-4.2)  {3};  
  \node (a4) at (5.9,-2.5)  {4};
  \node (ap-3) at (5.2,-0.9)  {$\frac{p-3}{2}$};
  \node (ap-1) at (3.85,-0.15)  {$\frac{p-1}{2}$};
  \node (ap+3)  at (2.15,-0.15)  {$\frac{p+3}{2}$} ;
  \node (ap+5) at (0.8,-0.9)  {$\frac{p+5}{2}$};
  \node (ap+7) at (0.1,-2.5)  {$\frac{p+7}{2}$};
  \node (p-2) at (0.3,-4.2)  {\small{$p-2$}};
  \node (p-1) at (1.5,-5.5)  {\small{$p-1$}};
  \node (a12) at (3,-3)  {\small{$\frac{p+1}{2}$}};

  
  \path (a4) -- node[auto=false, fill=none, sloped]{\ldots} (ap-3);
  \path (ap+7) -- node[auto=false, fill=none, sloped]{\ldots} (p-2);

\end{tikzpicture} \caption*{$\Omega$}
\end{minipage}
\caption*{Figure 1 }
\label{fig:1}
\end{table}

For each star:
$$S_i =  \{\{i,\infty\}, \{i, i +1\}, \{i, i+2\}, \dots, \{i, i + (p-3)/2 \} \},$$
we define 
$$f(S_i) = \{\{f(i),f(\infty)\}, \{f(i), f(i+1)\}, \{f(i), f(i+2)\}, \dots, \{f(i), f(i + (p-3)/2) \} \}.$$

Then the sets of stars $S_i$ for $1 \leq i \leq (p-1)/2$ are transformed as described in the next lemma.

\begin{lemma} \label{lem:fS_1_upto_(p-1)/2}
For $1 \leq i \leq (p-1)/2$, $f(S_i) = \{\{i, i + 1\}, \{i, i +2\},  \dots, \{i, i + (p-1)/2 \} \}$.
\end{lemma}

\begin{proof}
Observe that it is true for $i = 1 $. Now for $2 \leq i \leq (p-1)/2$, $S_i$ transforms as follows:
$$
\renewcommand{\arraystretch}{1.5}
\begin{array}{c ccc c}
 S_i & &\xrightarrow{f}&& f(S_i) \\
\{i, i+1\} & &&& \{i, i+1\} \\
\{i, i+2\} & &&& \{i, i+2\} \\
\vdots & &&& \vdots \\
\{i, (p-1)/2\} & &&& \{i, (p-1)/2 \} \\
\{i, \infty \} & &&& \{i, (p+1)/2 \} \\
\{i, (p+1)/2\} & &&& \{i, (p+3)/2 \} \\
\vdots & &&& \vdots \\
\{i, i + (p-3)/2\} & &&& \{i, i + (p-1)/2 \} \\
\end{array}
$$
Thus we have the required result.
\end{proof}

Consider the following permutation on the set $\Omega$.
$$ \rho(z) = \begin{cases}
 z & \textup{if } \hspace{0.5cm} z = (p+1)/2 \\
 1 & \textup{if } \hspace{0.5cm} z = (p-1)/2 \\
 (z + \frac{p+1}{2}) \pmod{p} & otherwise.
 \end{cases} 
 $$
 
By the definition of $f$ and $\rho$, the remaining set of stars,  $f(S_i)$ where $ (p+1)/2 \leq i \leq p-2$ can be expressed as follows.

\begin{lemma} \label{lem:rho_f_Si}
For $1 \leq i \leq (p-3)/2$, 
$  f(S_{i+ (p-1)/2}) = \rho(f(S_i))$.
\end{lemma}

\begin{proof}

Let $j = i + (p-1)/2$, where $1 \leq i \leq (p-3)/2$. Then $(p+1)/2 \leq j \leq p-2$. We want to show that $f(S_j) = \rho(f(S_i))$. By using the definitions of $f$ and $\rho$ the transformations are as follows:
$$
\renewcommand{\arraystretch}{1.5}
\begin{array}{ccc c ccc}
 S_j  & \xrightarrow{f} \ & f(S_j) & \ \hspace{0.5cm} \   & \rho(f(S_i)) & \xleftarrow{\rho} \ & f(S_i) \\
\{j, \infty\} & & \{j+1, (p+1)/2\}  &=& \{j+1, (p+1)/2\}&& \{i, (p+1)/2\} \\
\{j, j+1\} & & \{j+1, j+2\} &=& \{j+1, i+1 + (p+1)/2\}&& \{i, i+1\}  \\
\vdots & & \vdots && \vdots && \vdots  \\
\{j, p-3\} & & \{j+1, p-2\} &=& \{j+1, (p-5)/2 + (p+1)/2\} && \{i, (p-5)/2\}  \\
\{j, 0\} & & \{j+1, p-1\} &=& \{j+1, (p-3)/2 + (p+1)/2\} && \{i, (p-3)/2\} \\
\{j, 1\} & & \{j+1, 1\}  &=& \{j+1, 1\}&& \{i, (p-1)/2\} \\
\{j, 2\} & & \{j+1, 2\}  &=& \{j+1, (p+3)/2 + (p+1)/2 \}&& \{i, (p+3)/2\} \\
\vdots & & \vdots && \vdots && \vdots  \\
\{j, j + (p-3)/2\} && \{j+1, j + (p-3)/2\}  &=& \{j+1, i \} && \{i, i + (p-1)/2\} \\
\end{array}
$$
Hence the result.
\end{proof}

By combining Lemma \ref{lem:fS_1_upto_(p-1)/2} and Lemma \ref{lem:rho_f_Si} we have the following result. 

\begin{lemma} \label{lem:partition_graph}
    The set of stars $ \{f(S_i):  i \in K\} \cup \{ \rho(f(S_i)): 1 \leq i \leq (p-3)/2 \}  $ partitions the edge-set of the complete graph $G_{p-1}$ with the vertex set $\Omega$.
\end{lemma}

Now consider the set of sub-arrays $\{s_{(\alpha, h)}: \alpha \in \Omega, h \in K\}$ defined in (\ref{eq:intercalates}). Let $\{A_\alpha^\ast: \alpha \in \Omega\}$ be the set of arrays obtained by replacing $s_{(h+k, h)}$  with $\overline{s}_{(h+k, h)}$ in $A_{h+k}$ for each $h, k \in K.$
Then we have the following.

\begin{lemma} \label{lem:A_alpha_stars}
Let $\{A_\alpha^\ast: \alpha \in \Omega\}$ be the set of arrays described above.
\begin{enumerate}
    \item If $\{\alpha, \beta\} \in f(S_i)$ for some $ 1 \leq i \leq (p-1)/2 $, then:
    \begin{enumerate}
        \item[\textup{(a)}] $\vert A_\alpha^\ast A_\beta^\ast \vert_{(x,y)} = \vert A_\alpha A_\beta\vert_{(x,y)} + 1$, where $x,y \in \{0,1\}$ and $x \neq y$.
        \item[\textup{(b)}] $\vert A_\alpha^\ast A_\beta^\ast \vert_{(x,x)} = \vert A_\alpha A_\beta\vert_{(x,x)} - 1$, where $x \in \{0,1\}$.
    \end{enumerate}
    \item Otherwise: $\vert A_\alpha^\ast A_\beta^\ast \vert_{(x,y)} = \vert A_\alpha A_\beta\vert_{(x,y)} $.
\end{enumerate}
\end{lemma}

\begin{proof}
This follows by Lemma \ref{lem:intercalates}, Lemma \ref{lem:s_alphas_dont_overlap} and Lemma \ref{lem:trades}.
\end{proof}
Now we define another set of arrays $\{A_\alpha': \alpha \in \Omega\}$ that we will use in the fourth quadrant of our final arrays. Formally, let $\{A_\alpha': \alpha \in \Omega\}$ be the set of arrays obtained by replacing $s_{(h+k, h)}$  with $\overline{s}_{(h+k, h)}$ in $A_{h+k}$ for each $h \in \{ 1,2, \dots, \frac{p-3}{2} \}$ and $ k \in K.$ Observe that the array $A_\alpha' = A_\alpha^\ast $ for $\alpha \in \Omega \setminus \{(p+1)/2, \dots, p-1 \}$ and for $\alpha \in \{(p+1)/2, \dots, p-1 \}$, the only difference between $A_\alpha' $ and $ A_\alpha^\ast$ is that the array $A_\alpha'$ contains $s_{(\alpha, h)} $ and $  A_\alpha^\ast$ contains $\overline{s}_{(\alpha, h)} $ for $h = (p-1)/2$. Thus we have a similar result as Lemma \ref{lem:A_alpha_stars} for the set $\{A_\alpha': \alpha \in \Omega\}$.

\begin{lemma}\label{lem:A_alpha_primes}
Let $\{A_\alpha': \alpha \in \Omega\}$ be the set of arrays described above.
\begin{enumerate}
    \item If $\{\alpha, \beta\} \in f(S_i)$ for some $ 1 \leq i \leq (p-3)/2 $, then:
    \begin{enumerate}
        \item[\textup{(a)}] $\vert A_\alpha' A_\beta' \vert_{(x,y)} = \vert A_\alpha A_\beta\vert_{(x,y)} + 1$, where $x,y \in \{0,1\}$ and $x \neq y$.
        \item[\textup{(b)}] $\vert A_\alpha' A_\beta' \vert_{(x,x)} = \vert A_\alpha A_\beta\vert_{(x,x)} - 1$, where $x \in \{0,1\}$.
    \end{enumerate}
    \item Otherwise: $\vert A_\alpha' A_\beta' \vert_{(x,y)} = \vert A_\alpha A_\beta\vert_{(x,y)} $.
\end{enumerate}
\end{lemma}

Now we give our main result.

\begin{theorem}
Let $p \geq 3$ be a prime. Then there exists a set of $p-1$ binary MOFS of order $2p$.
\end{theorem}

\begin{proof}
Corresponding to each $\alpha \in \Omega$ we construct a binary frequency square $L_\alpha$ of order $2p$ as follows. 
\begin{equation} \label{eq:L_alpha}
F_\alpha = \left[
\renewcommand{\arraystretch}{1.2}
\begin{array}{c|c}
\hspace{0.3cm} A_\alpha^\ast \hspace{0.4cm} \ & \overline{A}_{\alpha} \\ \hline
\overline{A}_{\alpha} & A_{\rho^{-1}(\alpha)}' \\
\end{array} \right]
\end{equation}
Let $\alpha,\beta \in \Omega $. Let $\alpha' = \rho^{-1}(\alpha)$ and $\beta' = \rho^{-1}(\beta)$. Then, for all $x,y \in \{0,1\}$:
\begin{equation} \label{eq:L_alpha_L_beta}
\begin{split}
\vert F_{\alpha} F_{\beta} \vert_{(x,y)} = \vert A_\alpha^\ast A_{\beta}^\ast \vert_{(x,y)} + 2 \vert \overline{A}_{\alpha} \overline{A}_{\beta} \vert_{(x,y)} + \vert A_{\alpha'}' A_{\beta'}' \vert_{(x,y)}
\end{split}
\end{equation}
Now consider the following two cases:\\
   \textbf{Case I:} $\{\alpha, \beta\} \in f(S_i)$ for some $i \in K$.
   Then, by Lemma \ref{lem:partition_graph}, $\{\alpha, \beta\} \not \in \rho(f(S_i))$ for all $i \in \{1, \dots, (p-3)/2 \}$. This implies $\{\rho^{-1}(\alpha), \rho^{-1}(\beta)\} \not \in f(S_i)$ for all $i \in \{1, \dots, (p-3)/2 \}$. Thus by Lemma \ref{lem:A_alpha_stars} and Lemma \ref{lem:A_alpha_primes} we have:
$$
\renewcommand{\arraystretch}{1.5}
\begin{array}{cl}
\vert A_\alpha^\ast A_{\beta}^\ast \vert_{(x,y)}  = \vert A_\alpha A_{\beta} \vert_{(x,y)} + 1 & \textup{ whenever } x \neq y \\
\vert A_\alpha^\ast A_{\beta}^\ast \vert_{(x,y)} = \vert A_\alpha A_{\beta} \vert_{(x,y)} - 1 & \textup{ when } x = y \\
\vert A_{\alpha'}' A_{\beta'}' \vert_{(x,y)}  = \vert A_{\alpha'} A_{\beta'} \vert_{(x,y)} = \vert A_\alpha A_{\beta} \vert_{(x,y)} & \textup{ for all } x,y \in \{0,1\}.
\end{array}
$$

        Therefore, by using Lemma \ref{lem:order_pairs_in_A_alpha_A_beta} and Corollary \ref{cor:order_pairs_in_A_alpha_A_beta} in (\ref{eq:L_alpha_L_beta}) we get:
        $$\vert F_{\alpha} F_{\beta} \vert_{(x,y)} = p^2 \hspace{0.5cm} \textup{ for all } x,y \in \{0,1\}.$$

    \textbf{Case II:} $\{\alpha, \beta\} \not \in f(S_i)$ for all $i \in K$. 
    Then, by Lemma \ref{lem:partition_graph}, $\{\alpha, \beta\} \in \rho(f(S_i))$ for some $i \in \{1, \dots, (p-3)/2 \}$. This implies $\{\rho^{-1}(\alpha), \rho^{-1}(\beta)\} \in f(S_i)$ for some $i \in \{1, \dots, (p-3)/2 \}$. Thus by Lemma \ref{lem:A_alpha_stars} and Lemma \ref{lem:A_alpha_primes} we have:
$$
\renewcommand{\arraystretch}{1.5}
\begin{array}{cl}
\vert A_{\alpha'}' A_{\beta'}' \vert_{(x,y)}  = \vert A_{\alpha'} A_{\beta'} \vert_{(x,y)} + 1 = \vert A_\alpha A_{\beta} \vert_{(x,y)} + 1 & \textup{ whenever } x \neq y \\
\vert A_{\alpha'}' A_{\beta'}' \vert_{(x,y)}  = \vert A_{\alpha'} A_{\beta'} \vert_{(x,y)} - 1 = \vert A_\alpha A_{\beta} \vert_{(x,y)} - 1 & \textup{ when } x = y \\
\vert A_\alpha^\ast A_{\beta}^\ast \vert_{(x,y)}  = \vert A_\alpha A_{\beta} \vert_{(x,y)} & \textup{ for all } x,y \in \{0,1\}.
\end{array}
$$
             
        Therefore, again by using Lemma \ref{lem:order_pairs_in_A_alpha_A_beta} and Corollary \ref{cor:order_pairs_in_A_alpha_A_beta} in (\ref{eq:L_alpha_L_beta}) we get:
        $$\vert F_{\alpha} F_{\beta} \vert_{(x,y)} = p^2 \hspace{0.5cm} \textup{ for all } x,y \in \{0,1\}.$$

This completes the proof.
\end{proof}

Next, we include an example here to further illustrate the construction. 

\begin{example}
Let us construct a set of 6 binary MOFS of order 14. Here $p = 7$, $\Omega = \{ 1, \dots, 6 \}$, $K = \{1, 2, 3\}$ and ${\bf v} = (1,1,1,1,,0,0,0)$. The set of $A_\alpha$ for $\alpha \in \Omega$ is given in Table \ref{tbl:A_alpha_6MOFS(14)}. 

\begin{table}[H]
\centering
\begin{small}
\begin{tabular}{|cccc ccc|}
\hline
1&	1\cellcolor[gray]{0.8}&	1&	1&	0\cellcolor[gray]{0.8}&	0&	0\\
0&	1\cellcolor[gray]{0.8}&	1&	1&	1\cellcolor[gray]{0.8}&	0&	0\\
0&	0&	1&	1&	1&	1&	0\\
0&	0&	0&	1&	1&	1&	1\\
1&	0&	0&	0&	1&	1&	1\\
1&	1&	0&	0&	0&	1&	1\\
1&	1&	1&	0&	0&	0&	1\\\hline
\multicolumn{7}{c}{$A_1$} \\
\end{tabular}
\quad \quad
\begin{tabular}{|cccc ccc|}
\hline
1&	1\cellcolor[gray]{0.8}&	1&	1&	0\cellcolor[gray]{0.8}&	0&	0\\
0&	0\cellcolor[gray]{0.8}&	1&	1&	1\cellcolor[gray]{0.8}&	1&	0\\
1&	0&	0&	0&	1&	1&	1\\
1&	1&	1&	0&	0&	0&	1\\
0&	1&	1&	1&	1&	0&	0\\
0&	0&	0&	1&	1&	1&	1\\
1&	1&	0&	0&	0&	1&	1\\
\hline
\multicolumn{7}{c}{$A_2$} \\
\end{tabular}
\quad \quad
\begin{tabular}{|cccc ccc|}
\hline
1&	1\cellcolor[gray]{0.8}&	1&	1&	0\cellcolor[gray]{0.8}&	0&	0\\
0&	0\cellcolor[gray]{0.8}&	0&	1&	1\cellcolor[gray]{0.8}&	1&	1\\
1&	1&	1&	0&	0&	0&	1\\
0&	0&	1&	1&	1&	1&	0\\
1&	1&	0&	0&	0&	1&	1\\
0&	1&	1&	1&	1&	0&	0\\
1&	0&	0&	0&	1&	1&	1\\
\hline
\multicolumn{7}{c}{$A_3$} \\
\end{tabular}

\vspace{0.5cm}
\begin{tabular}{|cccc ccc|}
\hline
1&	1\cellcolor[gray]{0.8}&	1&	1&	0\cellcolor[gray]{0.8}&	0&	0\\
1&	0\cellcolor[gray]{0.8}&	0&	0&	1\cellcolor[gray]{0.8}&	1&	1\\
0&	1&	1&	1&	1&	0&	0\\
1&	1&	0&	0&	0&	1&	1\\
0&	0&	1&	1&	1&	1&	0\\
1&	1&	1&	0&	0&	0&	1\\
0&	0&	0&	1&	1&	1&	1\\
\hline
\multicolumn{7}{c}{$A_4$} \\
\end{tabular}
\quad \quad
\begin{tabular}{|cccc ccc|}
\hline
1&	1\cellcolor[gray]{0.8}&	1&	1&	0\cellcolor[gray]{0.8}&	0&	0\\
1&	1\cellcolor[gray]{0.8}&	0&	0&	0\cellcolor[gray]{0.8}&	1&	1\\
0&	0&	0&	1&	1&	1&	1\\
0&	1&	1&	1&	1&	0&	0\\
1&	1&	1&	0&	0&	0&	1\\
1&	0&	0&	0&	1&	1&	1\\
0&	0&	1&	1&	1&	1&	0\\
\hline
\multicolumn{7}{c}{$A_5$} \\
\end{tabular}
\quad \quad
\begin{tabular}{|cccc ccc|}
\hline
1&	1\cellcolor[gray]{0.8}&	1&	1&	0\cellcolor[gray]{0.8}&	0&	0\\
1&	1\cellcolor[gray]{0.8}&	1&	0&	0\cellcolor[gray]{0.8}&	0&	1\\
1&	1&	0&	0&	0&	1&	1\\
1&	0&	0&	0&	1&	1&	1\\
0&	0&	0&	1&	1&	1&	1\\
0&	0&	1&	1&	1&	1&	0\\
0&	1&	1&	1&	1&	0&	0\\\hline
\multicolumn{7}{c}{$A_6$} \\
\end{tabular}
\end{small}
\caption{Set of $A_\alpha$ to construct $6$--MOFS(14)}
\label{tbl:A_alpha_6MOFS(14)}
\end{table}

Now consider the set of coincident arrays $\{s_{(\alpha, 1)}: \alpha \in \Omega\}$ described in (\ref{eq:intercalates}), shown as highlighted cells  in Table \ref{tbl:A_alpha_6MOFS(14)}. 
Then $A_1$ contains $\bigl( \begin{smallmatrix} 1 & 0 \\ 1 & 1 \end{smallmatrix} \bigr)$, $A_2, \dots, A_4$ contain $\bigl( \begin{smallmatrix} 1 & 0 \\ 0 & 1 \end{smallmatrix} \bigr)$ and the rest of $A_\alpha$ contain $\bigl( \begin{smallmatrix} 1 & 0 \\ 1 & 0 \end{smallmatrix} \bigr)$ at this position.
We replace $s_{(\alpha, 1)} = \bigl( \begin{smallmatrix} 1 & 0 \\ 0 & 1 \end{smallmatrix} \bigr)$ with $\overline{s}_{(\alpha, 1)} = \bigl( \begin{smallmatrix} 0 & 1 \\ 1 & 0 \end{smallmatrix} \bigr)$ in $A_\alpha $ for $\alpha \in \{2, 3, 4\}$.
Similarly, by replacing $s_{(h+k, h)}$  with $\overline{s}_{(h+k, h)}$ in $A_{h+k}$ for $h \in \{ 2,3\}$ and $ k \in K$, we get the set  $\{ A_\alpha^\ast : \alpha \in \Omega \}$. The first two rows of $A_\alpha^\ast, \alpha \in \Omega$ are shown in Table \ref{tbl:A_alpha_star_6MOFS(14)}. 

\begin{table}[H]
\centering
\begin{small}
\begin{tabular}{|cccc ccc|}
\hline
1&	1&	1&	1&	0&	0&	0\\
0&	1&	1&	1&	1&	0&	0\\\hline
\multicolumn{7}{c}{$A_1^\ast$} \\
\end{tabular}
\quad \quad
\begin{tabular}{|cccc ccc|}
\hline
1&	0&	1&	1&	1&	0&	0\\
0&	1&	1&	1&	0&	1&	0\\
\hline
\multicolumn{7}{c}{$A_2^\ast$} \\
\end{tabular}
\quad \quad
\begin{tabular}{|cccc ccc|}
\hline
1&	0&	0&	1&	1&	1&	0\\
0&	1&	1&	1&	0&	0&	1\\
\hline
\multicolumn{7}{c}{$A_3^\ast$} \\
\end{tabular}

\vspace{0.5cm}
\begin{tabular}{|cccc ccc|}
\hline
1&	0&	0&	0\cellcolor[gray]{0.8}&	1&	1&	1\cellcolor[gray]{0.8}\\
1&	1&	1&	1\cellcolor[gray]{0.8}&	0&	0&	0\cellcolor[gray]{0.8}\\
\hline
\multicolumn{7}{c}{$A_4^\ast$} \\
\end{tabular}
\quad \quad
\begin{tabular}{|cccc ccc|}
\hline
1&	1&	0&	0\cellcolor[gray]{0.8}&	0&	1&	1\cellcolor[gray]{0.8}\\
1&	1&	1&	1\cellcolor[gray]{0.8}&	0&	0&	0\cellcolor[gray]{0.8}\\
\hline
\multicolumn{7}{c}{$A_5^\ast$} \\
\end{tabular}
\quad \quad
\begin{tabular}{|cccc ccc|}
\hline
1&	1&	1&	0\cellcolor[gray]{0.8}&	0&	0&	1\cellcolor[gray]{0.8}\\
1&	1&	1&	1\cellcolor[gray]{0.8}&	0&	0&	0\cellcolor[gray]{0.8}\\\hline
\multicolumn{7}{c}{$A_6^\ast$} \\
\end{tabular}
\end{small}
\caption{First two rows of $A_\alpha^\ast$ to construct $6$--MOFS(14)}
\label{tbl:A_alpha_star_6MOFS(14)}
\end{table}

Now to obtain the set $\{A_\alpha' : \alpha \in \Omega \}$, we repeat the same procedure as above only this time we do not replace $s_{(h+k, h)}$ with $\overline{s}_{(h+k, h)}$ when $h = (p-1)/2$. The first two rows of  $A_\alpha ' $ are shown in Table \ref{tbl:A_alpha_prime_6MOFS(14)}. Observe the difference between $A_{\alpha}^\ast$ and $A_{\alpha}'$ for $\alpha \in \{4,5,6\}$ in the highlighted cells of Table \ref{tbl:A_alpha_star_6MOFS(14)} and Table \ref{tbl:A_alpha_prime_6MOFS(14)}.

\begin{table}[H]
\centering
\begin{small}
\begin{tabular}{|cccc ccc|}
\hline
1&	1&	1&	1&	0&	0&	0\\
0&	1&	1&	1&	1&	0&	0\\\hline
\multicolumn{7}{c}{$A_1'$} \\
\end{tabular}
\quad \quad
\begin{tabular}{|cccc ccc|}
\hline
1&	0&	1&	1&	1&	0&	0\\
0&	1&	1&	1&	0&	1&	0\\
\hline
\multicolumn{7}{c}{$A_2'$} \\
\end{tabular}
\quad \quad
\begin{tabular}{|cccc ccc|}
\hline
1&	0&	0&	1&	1&	1&	0\\
0&	1&	1&	1&	0&	0&	1\\
\hline
\multicolumn{7}{c}{$A_3'$} \\
\end{tabular}

\vspace{0.5cm}
\begin{tabular}{|cccc ccc|}
\hline
1&	0&	0&	1\cellcolor[gray]{0.8}&	1&	1&	0\cellcolor[gray]{0.8}\\
1&	1&	1&	0\cellcolor[gray]{0.8}&	0&	0&	1\cellcolor[gray]{0.8}\\
\hline
\multicolumn{7}{c}{$A_4'$} \\
\end{tabular}
\quad \quad
\begin{tabular}{|cccc ccc|}
\hline
1&	1&	0&	1\cellcolor[gray]{0.8}&	0&	1&	0\cellcolor[gray]{0.8}\\
1&	1&	1&	0\cellcolor[gray]{0.8}&	0&	0&	1\cellcolor[gray]{0.8}\\
\hline
\multicolumn{7}{c}{$A_5'$} \\
\end{tabular}
\quad \quad
\begin{tabular}{|cccc ccc|}
\hline
1&	1&	1&	1\cellcolor[gray]{0.8}&	0&	0&	0\cellcolor[gray]{0.8}\\
1&	1&	1&	0\cellcolor[gray]{0.8}&	0&	0&	1\cellcolor[gray]{0.8}\\\hline
\multicolumn{7}{c}{$A_6'$} \\
\end{tabular}
\end{small}
\caption{First two rows of $A_\alpha'$ to construct $6$--MOFS(14)}
\label{tbl:A_alpha_prime_6MOFS(14)}
\end{table}

Now consider the permutation $\rho$ on the set $\Omega$. 
$$
\rho = \left( 
\begin{array}{ccc ccc} 
1 & 2 & 3 & 4 & 5 & 6 \\
5 & 6 & 1 & 4 & 2 & 3 \\
\end{array}
\right)
$$
Thus we get a set $\{F_\alpha : \alpha \in \Omega\}$ of $6$ binary MOFS of order 14, where $F_\alpha$ is of the form:
$$ 
F_\alpha = \left[
\renewcommand{\arraystretch}{1.2}
\begin{array}{c|c}
\hspace{0.3cm} A_\alpha^\ast \hspace{0.4cm} \ & \overline{A}_{\alpha} \\ \hline
\overline{A}_{\alpha} & A_{\rho^{-1}(\alpha)}' \\
\end{array} \right]
$$
A complete description of each frequency square $F_\alpha$ is given in Appendix \ref{sec:appendix_10_MOFS}.

\end{example}

\bibliographystyle{abbrv}
\bibliography{References21}

\begin{thebibliography}{10}

\bibitem{ball2012sets}
S.~Ball.
\newblock On sets of vectors of a finite vector space in which every subset of
  basis size is a basis.
\newblock {\em Journal of the European Mathematical Society}, 14:733--748,
  2012.

\bibitem{ball2012setsII}
S.~Ball and J.~De~Beule.
\newblock On sets of vectors of a finite vector space in which every subset of
  basis size is a basis ii.
\newblock {\em Designs, Codes and Cryptography}, 65:5--14, 2012.

\bibitem{britz2020mutually}
T.~Britz, N.~J. Cavenagh, A.~Mammoliti, and I.~M. Wanless.
\newblock Mutually orthogonal binary frequency squares.
\newblock {\em The electronic journal of combinatorics}, 27(3), 2020.

\bibitem{cheng1980orthogonal}
C.-S. Cheng.
\newblock Orthogonal arrays with variable numbers of symbols.
\newblock {\em The Annals of Statistics}, 8:447--453, 1980.

\bibitem{Colbourn:2006:HCD:1202540}
C.~J. Colbourn and J.~H. Dinitz.
\newblock {\em Handbook of Combinatorial Designs, Second Edition (Discrete
  Mathematics and Its Applications)}.
\newblock Chapman \& Hall/CRC, 2006.

\bibitem{damelin2004number}
S.~Damelin, G.~Michalski, G.~Mullen, and D.~Stone.
\newblock The number of linearly independent binary vectors with applications
  to the construction of hypercubes and orthogonal arrays, pseudo $(t, m,
  s)$-nets and linear codes.
\newblock {\em Monatshefte f{\"u}r Mathematik}, 141(4):277--288, 2004.

\bibitem{damelin2007cardinality}
S.~Damelin, G.~Michalski, and G.~L. Mullen.
\newblock The cardinality of sets of $k$-independent vectors over finite
  fields.
\newblock {\em Monatshefte f{\"u}r Mathematik}, 150(4):289--295, 2007.

\bibitem{djokovic2007hadamard}
D.~Z. Djokovic.
\newblock Hadamard matrices of order 764 exist.
\newblock {\em Combinatorica}, 28:487--489, 2008.

\bibitem{federer1984pairwise}
W.~Federer, A.~Hedayat, and J.~Mandeli.
\newblock Pairwise orthogonal f-rectangle designs.
\newblock {\em Journal of statistical planning and inference}, 10(3):365--374,
  1984.

\bibitem{federer1977existence}
W.~T. Federer et~al.
\newblock On the existence and construction of a complete set of orthogonal $ f
  (4t; 2t, 2t) $-squares design.
\newblock {\em The Annals of Statistics}, 5(3):561--564, 1977.

\bibitem{markus2009codestable}
M.~Grassl.
\newblock Code tables: Bounds on the parameters of various types of codes.
\newblock Updated Aug 2022.
\newblock http://codetables.markus-grassl.de/.

\bibitem{hedayat1975further}
A.~Hedayat, D.~Raghavarao, E.~Seiden, et~al.
\newblock Further contributions to the theory of $ f $-squares design.
\newblock {\em The Annals of Statistics}, 3(3):712--716, 1975.

\bibitem{hedayat1999orthogonal}
S.~Hedayat, Sloane.
\newblock {\em Orthogonal Arrays}.
\newblock Springer, New York, NY, 1999.

\bibitem{horadam2012hadamard}
K.~J. Horadam.
\newblock {\em Hadamard matrices and their applications}.
\newblock Princeton university press, 2012.

\bibitem{laywine2001table}
C.~F. Laywine and G.~L. Mullen.
\newblock A table of lower bounds for the number of mutually orthogonal
  frequency squares.
\newblock {\em Ars Combinatoria}, 59:85--96, 2001.

\bibitem{li2014some}
M.~Li, Y.~Zhang, and B.~Du.
\newblock Some new results on mutually orthogonal frequency squares.
\newblock {\em Discrete Mathematics}, 331:175--187, 2014.

\bibitem{mandeli1992complete}
J.~Mandeli.
\newblock Complete-sets of mutually orthogonal frequency rectangle designs
  having twice a prime power number of columns.
\newblock {\em Utilas Mathematica}, 41:151--160, 1992.

\bibitem{Mandeli1984ontheconstruction}
J.~Mandeli and W.~Federer.
\newblock On the construction of mutually orthogonal f-hyperrectangles.
\newblock {\em Utilitas Math.}, 25:315--–324, 1984.

\bibitem{mavron2000frequency}
V.~C. Mavron.
\newblock Frequency squares and affine designs.
\newblock {\em The Electronic Journal of Combinatorics}, 7(1):56, 2000.

\bibitem{rahim2021row}
F.~Rahim and N.~J. Cavenagh.
\newblock Row-column factorial designs with multiple levels.
\newblock {\em Journal of Combinatorial Designs}, 29:750--764, 2021.

\bibitem{rahim2022row}
F.~Rahim and N.~J. Cavenagh.
\newblock Row-column factorial designs with strength at least $2$.
\newblock {\em arXiv preprint arXiv:2207.02397}, 2022.

\bibitem{street1979generalized}
D.~Street.
\newblock Generalized hadamard matrices, orthogonal arrays and f-squares.
\newblock {\em Ars Combinatoria}, 8:131--141, 1979.

\bibitem{tassa2009proper}
T.~Tassa and J.~L. Villar.
\newblock On proper secrets,(t, k)-bases and linear codes.
\newblock {\em Designs, Codes and Cryptography}, 52(2):129--154, 2009.

\end{thebibliography}

\appendix

\newpage

\section{6 binary MOFS(14)} \label{sec:appendix_10_MOFS}

\begin{table}[H]
    \centering
    \scriptsize
\renewcommand{\arraystretch}{1.25}
\begin{tabular}{|cccc ccc|cccc ccc|}
\hline
1&	1&	1&	1&	0&	0&	0&	0&	0&	0&	0&	1&	1&	1\\
0&	1&	1&	1&	1&	0&	0&	1&	0&	0&	0&	0&	1&	1\\
0&	0&	1&	1&	1&	1&	0&	1&	1&	0&	0&	0&	0&	1\\
0&	0&	0&	1&	1&	1&	1&	1&	1&	1&	0&	0&	0&	0\\
1&	0&	0&	0&	1&	1&	1&	0&	1&	1&	1&	0&	0&	0\\
1&	1&	0&	0&	0&	1&	1&	0&	0&	1&	1&	1&	0&	0\\
1&	1&	1&	0&	0&	0&	1&	0&	0&	0&	1&	1&	1&	0\\\hline
0&	0&	0&	0&	1&	1&	1&	1&	0&	0&	1&	1&	1&	0\\
1&	0&	0&	0&	0&	1&	1&	0&	1&	1&	1&	0&	0&	1\\
1&	1&	0&	0&	0&	0&	1&	1&	1&	1&	0&	0&	0&	1\\
1&	1&	1&	0&	0&	0&	0&	0&	0&	1&	1&	1&	1&	0\\
0&	1&	1&	1&	0&	0&	0&	1&	1&	0&	0&	0&	1&	1\\
0&	0&	1&	1&	1&	0&	0&	0&	1&	1&	1&	1&	0&	0\\
0&	0&	0&	1&	1&	1&	0&	1&	0&	0&	0&	1&	1&	1\\
\hline
\multicolumn{14}{c}{} \\
\multicolumn{14}{c}{\Large $F_1$} \\
\end{tabular}
\quad \quad
\begin{tabular}{|cccc ccc|cccc ccc|}
\hline
1&	0&	1&	1&	1&	0&	0&	0&	0&	0&	0&	1&	1&	1\\
0&	1&	1&	1&	0&	1&	0&	1&	1&	0&	0&	0&	0&	1\\
1&	0&	0&	0&	1&	1&	1&	0&	1&	1&	1&	0&	0&	0\\
1&	1&	1&	0&	0&	0&	1&	0&	0&	0&	1&	1&	1&	0\\
0&	1&	1&	1&	1&	0&	0&	1&	0&	0&	0&	0&	1&	1\\
0&	0&	0&	1&	1&	1&	1&	1&	1&	1&	0&	0&	0&	0\\
1&	1&	0&	0&	0&	1&	1&	0&	0&	1&	1&	1&	0&	0\\\hline
0&	0&	0&	0&	1&	1&	1&	1&	1&	0&	1&	0&	1&	0\\
1&	1&	0&	0&	0&	0&	1&	1&	1&	1&	0&	0&	0&	1\\
0&	1&	1&	1&	0&	0&	0&	0&	0&	0&	1&	1&	1&	1\\
0&	0&	0&	1&	1&	1&	0&	0&	1&	1&	1&	1&	0&	0\\
1&	0&	0&	0&	0&	1&	1&	1&	1&	1&	0&	0&	0&	1\\
1&	1&	1&	0&	0&	0&	0&	1&	0&	0&	0&	1&	1&	1\\
0&	0&	1&	1&	1&	0&	0&	0&	0&	1&	1&	1&	1&	0\\
\hline
\multicolumn{14}{c}{} \\
\multicolumn{14}{c}{\Large $F_2$} \\
\end{tabular}

\vspace{0.7cm}

\begin{tabular}{|cccc ccc|cccc ccc|}
\hline
1&	0&	0&	1&	1&	1&	0&	0&	0&	0&	0&	1&	1&	1\\
0&	1&	1&	1&	0&	0&	1&	1&	1&	1&	0&	0&	0&	0\\
1&	1&	1&	0&	0&	0&	1&	0&	0&	0&	1&	1&	1&	0\\
0&	0&	1&	1&	1&	1&	0&	1&	1&	0&	0&	0&	0&	1\\
1&	1&	0&	0&	0&	1&	1&	0&	0&	1&	1&	1&	0&	0\\
0&	1&	1&	1&	1&	0&	0&	1&	0&	0&	0&	0&	1&	1\\
1&	0&	0&	0&	1&	1&	1&	0&	1&	1&	1&	0&	0&	0\\\hline
0&	0&	0&	0&	1&	1&	1&	1&	1&	1&	1&	0&	0&	0\\
1&	1&	1&	0&	0&	0&	0&	1&	1&	1&	0&	0&	0&	1\\
0&	0&	0&	1&	1&	1&	0&	1&	1&	0&	0&	0&	1&	1\\
1&	1&	0&	0&	0&	0&	1&	1&	0&	0&	0&	1&	1&	1\\
0&	0&	1&	1&	1&	0&	0&	0&	0&	0&	1&	1&	1&	1\\
1&	0&	0&	0&	0&	1&	1&	0&	0&	1&	1&	1&	1&	0\\
0&	1&	1&	1&	0&	0&	0&	0&	1&	1&	1&	1&	0&	0\\
\hline
\multicolumn{14}{c}{} \\
\multicolumn{14}{c}{\Large $F_3$} \\
\end{tabular}
\quad \quad
\begin{tabular}{|cccc ccc|cccc ccc|}
\hline
1&	0&	0&	0&	1&	1&	1&	0&	0&	0&	0&	1&	1&	1\\
1&	1&	1&	1&	0&	0&	0&	0&	1&	1&	1&	0&	0&	0\\
0&	1&	1&	1&	1&	0&	0&	1&	0&	0&	0&	0&	1&	1\\
1&	1&	0&	0&	0&	1&	1&	0&	0&	1&	1&	1&	0&	0\\
0&	0&	1&	1&	1&	1&	0&	1&	1&	0&	0&	0&	0&	1\\
1&	1&	1&	0&	0&	0&	1&	0&	0&	0&	1&	1&	1&	0\\
0&	0&	0&	1&	1&	1&	1&	1&	1&	1&	0&	0&	0&	0\\\hline
0&	0&	0&	0&	1&	1&	1&	1&	0&	0&	1&	1&	1&	0\\
0&	1&	1&	1&	0&	0&	0&	1&	1&	1&	0&	0&	0&	1\\
1&	0&	0&	0&	0&	1&	1&	0&	1&	1&	1&	1&	0&	0\\
0&	0&	1&	1&	1&	0&	0&	1&	1&	0&	0&	0&	1&	1\\
1&	1&	0&	0&	0&	0&	1&	0&	0&	1&	1&	1&	1&	0\\
0&	0&	0&	1&	1&	1&	0&	1&	1&	1&	0&	0&	0&	1\\
1&	1&	1&	0&	0&	0&	0&	0&	0&	0&	1&	1&	1&	1\\
\hline
\multicolumn{14}{c}{} \\
\multicolumn{14}{c}{\Large $F_4$} \\
\end{tabular}

\vspace{0.7cm}

\begin{tabular}{|cccc ccc|cccc ccc|}
\hline
1&	1&	0&	0&	0&	1&	1&	0&	0&	0&	0&	1&	1&	1\\
1&	1&	1&	1&	0&	0&	0&	0&	0&	1&	1&	1&	0&	0\\
0&	0&	0&	1&	1&	1&	1&	1&	1&	1&	0&	0&	0&	0\\
0&	1&	1&	1&	1&	0&	0&	1&	0&	0&	0&	0&	1&	1\\
1&	1&	1&	0&	0&	0&	1&	0&	0&	0&	1&	1&	1&	0\\
1&	0&	0&	0&	1&	1&	1&	0&	1&	1&	1&	0&	0&	0\\
0&	0&	1&	1&	1&	1&	0&	1&	1&	0&	0&	0&	0&	1\\\hline
0&	0&	0&	0&	1&	1&	1&	1&	1&	1&	1&	0&	0&	0\\
0&	0&	1&	1&	1&	0&	0&	0&	1&	1&	1&	1&	0&	0\\
1&	1&	1&	0&	0&	0&	0&	0&	0&	1&	1&	1&	1&	0\\
1&	0&	0&	0&	0&	1&	1&	0&	0&	0&	1&	1&	1&	1\\
0&	0&	0&	1&	1&	1&	0&	1&	0&	0&	0&	1&	1&	1\\
0&	1&	1&	1&	0&	0&	0&	1&	1&	0&	0&	0&	1&	1\\
1&	1&	0&	0&	0&	0&	1&	1&	1&	1&	0&	0&	0&	1\\
\hline
\multicolumn{14}{c}{} \\
\multicolumn{14}{c}{\Large $F_5$} \\
\end{tabular}
\quad \quad
\begin{tabular}{|cccc ccc|cccc ccc|}
\hline
1&	1&	1&	0&	0&	0&	1&	0&	0&	0&	0&	1&	1&	1\\
1&	1&	1&	1&	0&	0&	0&	0&	0&	0&	1&	1&	1&	0\\
1&	1&	0&	0&	0&	1&	1&	0&	0&	1&	1&	1&	0&	0\\
1&	0&	0&	0&	1&	1&	1&	0&	1&	1&	1&	0&	0&	0\\
0&	0&	0&	1&	1&	1&	1&	1&	1&	1&	0&	0&	0&	0\\
0&	0&	1&	1&	1&	1&	0&	1&	1&	0&	0&	0&	0&	1\\
0&	1&	1&	1&	1&	0&	0&	1&	0&	0&	0&	0&	1&	1\\\hline
0&	0&	0&	0&	1&	1&	1&	1&	0&	1&	1&	1&	0&	0\\
0&	0&	0&	1&	1&	1&	0&	0&	1&	1&	1&	0&	1&	0\\
0&	0&	1&	1&	1&	0&	0&	1&	0&	0&	0&	1&	1&	1\\
0&	1&	1&	1&	0&	0&	0&	1&	1&	1&	0&	0&	0&	1\\
1&	1&	1&	0&	0&	0&	0&	0&	1&	1&	1&	1&	0&	0\\
1&	1&	0&	0&	0&	0&	1&	0&	0&	0&	1&	1&	1&	1\\
1&	0&	0&	0&	0&	1&	1&	1&	1&	0&	0&	0&	1&	1\\
\hline
\multicolumn{14}{c}{} \\
\multicolumn{14}{c}{\Large $F_6$} \\
\end{tabular}
\caption{A set of 6 binary MOFS of order 14.}
\end{table}

\newpage
\section{Eigenvalues} \label{app:eigenvalues}
The eigenvalues of the matrix $M^TM$ in the proof of Theorem \ref{thm:upper_bound_rectangles} can be obtained from Lemma \ref{lem:eigenvalues_N} by substituting $c = n\lambda$ and $d = \lambda \lambda'$, where $\lambda  = m/q$ and $\lambda' = n/q$.
\begin{lemma} \label{lem:eigenvalues}
Let $H = aI_q + bJ_q$ be a $q \times q$ matrix, where $I_q$ is an identity matrix of order $q$ and $J_q$ is a matrix of ones of order $q$. Then $a + bq$ and $a$ are eigenvalues of $H$ with multiplicities 1 and $q-1$ respectively.
\end{lemma}

\begin{proof}
Observe that the matrix,
$$ H = 
\begin{pmatrix}
    a+b & b & \dots & b \\
    b & a+b & \dots & b \\
    \vdots & \vdots & \ddots & \vdots \\
    b & b & \dots & a+b \\
\end{pmatrix}
$$
can be reduced to the following lower triangular matrix:
$$ 
\begin{pmatrix}
    a+qb & 0 &\hspace{0.3cm}0& \dots & 0 \\
    b & a &\hspace{0.3cm}0& \dots & 0 \\
    b & 0 &\hspace{0.3cm}a& \dots & 0 \\
    \vdots & \vdots&\hspace{0.3cm} \vdots & \ddots & \vdots \\
    b & 0 &\hspace{0.3cm}0& \dots & a \\
\end{pmatrix}
$$
Thus the eigenvalues $a+qb$ and $a$ have multiplicities 1 and $q-1$, respectively. 
\end{proof}

\begin{lemma} \label{lem:eigenvalues_N}
Let $N$ be a $(kq) \times (kq)$ matrix of the following form:    
$$
\renewcommand{\arraystretch}{1.25}
N = \begin{pmatrix}
c I_q  &  d J_q  &  \ldots  & d J_q \\
d J_q &  c I_q   &  \ldots  &  d J_q  \\
\vdots        &    \vdots       &  \ddots  &   \vdots        \\
d J_q &  d J_q  &  \ldots  &  c I_q  \\
\end{pmatrix},
$$
where $I_q$ and $J_q$ are defined in Lemma \ref{lem:eigenvalues}. Then $N$ has eigenvalues $c + q(k-1)d$, $ c- qd,$ and $c$ with multiplicities 1, $k-1,$ and $k(q-1)$ respectively.
\end{lemma}

\begin{proof}
By row operations we can row-reduce $N$ to the following matrices (where ${\bf 0}$ is a $q \times q$ matrix of zeroes): 
$$
\renewcommand{\arraystretch}{1.5}
 \begin{pmatrix}
c I_q  &  d J_q - c I_q &  d J_q - c I_q & \ldots  & d J_q - c I_q\\
d J_q &  c I_q - d J_q  & {\bf 0}  &  \ldots  &  {\bf 0} \\
d J_q &  {\bf 0}   & c I_q - d J_q &  \ldots  &  {\bf 0} \\
\vdots        &    \vdots       &  \vdots  &   \ddots & \vdots        \\
d J_q &   {\bf 0}   & {\bf 0}  &  \ldots  &  c I_q - d J_q\\
\end{pmatrix},
$$

$$
\renewcommand{\arraystretch}{1.5}
\begin{pmatrix}
c I_q + (k-1) d J_q  &  {\bf 0}  &  {\bf 0}  & \ldots  & {\bf 0} \\
d J_q &  c I_q - d J_q  & {\bf 0}  &  \ldots  &  {\bf 0} \\
d J_q &  {\bf 0}   & c I_q - d J_q &  \ldots  &  {\bf 0} \\
\vdots        &    \vdots       &  \vdots  &   \ddots & \vdots        \\
d J_q &   {\bf 0}   & {\bf 0}  &  \ldots  &  c I_q - d J_q\\
\end{pmatrix}.
$$
Now by using Lemma \ref{lem:eigenvalues} the matrix $c I_q + (k-1) d J_q$ has eigenvalues $c+q(k-1)d$ and $c$ with multiplicities 1 and $q-1$, respectively. And the matrix $ c I_q - d J_q $ has eigenvalues $c - qd$ and $c$ with multiplicities 1 and $q-1$, respectively. Consequently the matrix $N$ has eigenvalues $c + q(k-1)d$, $ c- qd$, and $c$ with multiplicities 1, $k-1$, and $k(q-1)$ respectively.
\end{proof}

\end{document}